\documentclass[11pt]{amsart}

\usepackage{amsmath}
\usepackage{amssymb}

\newcommand{\dst}{\displaystyle}
\newtheorem{thm}{Theorem}[section]

\newtheorem{lem}[thm]{Lemma}
\newtheorem{cor}[thm]{Corollary}
\newtheorem{defi}[thm]{Definition}
\newtheorem*{A}{Gonchar's Theorem}

\begin{document}

\title[Direct and inverse results]
{Direct and inverse results on row sequences of Hermite-Pad\'e
approximants}

\author[J. Cacoq]{J. Cacoq}
\address{Dpto. de Matem\'aticas\\
Escuela Polit\'ecnica Superior \\
Universidad Carlos III de Madrid \\
Universidad 30, 28911 Legan\'es, Spain} \email{jcacoq@math.uc3m.es}
\thanks{The work of  B. de la Calle
received support from MINCINN under grant MTM2009-14668-C02-02 and
from UPM through Research Group ``Constructive Approximation Theory
and Applications". The work of  J. Cacoq and G. L\'opez was
supported by MINCINN under grant MTM2009-12740-C03-01}

\author[B. de la Calle]{\hspace*{0.1 cm} B. de la Calle
Ysern}\address{Dpto. de Matem\'atica Aplicada\\
E. T. S.  de Ingenieros Industriales\\
Universidad Polit\'ecnica de Madrid\\
Jos\'e G. Abascal 2, 28006 Madrid, Spain}
\email{bcalle@etsii.upm.es}

\author[G. L\'opez ]{G. L\'opez Lagomasino}
\address{Dpto. de Matem\'aticas\\
Escuela Polit\'ecnica Superior \\
Universidad Carlos III de Madrid \\
Universidad 30, 28911 Legan\'es, Spain} \email{lago@math.uc3m.es}

\keywords{Montessus de Ballore Theorem, simultaneous approximation,
Hermite-Pad\'e approximation, inverse results}

\subjclass[2010]{Primary 41A21, 41A28; Secondary 41A25, 41A27}

\dedicatory{Dedicated to A.A. Gonchar, on the occasion of his
eightieth birthday}

\begin{abstract} We give necessary and sufficient
conditions for the convergence with geometric rate of the  common
denominators of simultaneous rational interpolants with a bounded
number of poles. The conditions are expressed in terms of  intrinsic
properties of the system of functions used to build the
approximants. Exact rates of convergence for these denominators and
the simultaneous rational approximants are  provided.
\end{abstract}
\date{\today}
\maketitle


\section{Introduction}
 Let ${\bf f} = (f_1,\ldots,f_d)$ be a system
of $d$ formal or convergent Taylor expansions about the origin; that
is, for each $k=1,\dots,d$, we have
\begin{equation} \label{sistema} f_k(z) = \sum_{n=0}^{\infty} \phi_{n,k} z^n,
\qquad \phi_{n,k} \in {\mathbb{C}}.
\end{equation}
Let  $\mathbf{D}=\left(D_1,\dots,D_d\right)$ be a system of domains
such that, for each $k=1,\dots,d,$ $f_k$ is meromorphic in $D_k$. We
say that the point $\xi$  is a pole of ${\bf f}$ in $\mathbf{D}$ of
order $\tau$ if there exists an index $k\in\{1,\dots,d\}$  such that
$\xi\in D_k$ and it is a pole of $f_k$ of order $\tau$, and for
$j\not = k$ either $\xi$ is a pole of $f_j$ of order less than or
equal to $\tau$ or $\xi \not\in D_j$. When
$\mathbf{D}={(D,\dots,D)}$ we say that  $\xi$ is a pole of ${\bf f}$
in $D$.

 Let $R_0({\bf f})$ be
the largest disk in which all the expansions $f_k, k=1,\ldots,d$
correspond to analytic functions. If $R_0({\bf f}) =0$, we take
$D_{m}({\bf f}) = \emptyset, m \in {\mathbb{Z}}_+$; otherwise,
$R_m({\bf f})$ is the radius of the largest disk $D_{m}({\bf f})$
centered at the origin to which all the analytic elements $(f_k,
D_0(f_k))$ can be extended so that ${\bf f}$ has at most $m$ poles
counting multiplicities. The disk $D_{m}({\bf f})$ constitutes for
systems of functions the analogue of the $m$-th disk of meromorphy
defined by J. Hadamard in \cite{Had} for $d=1$. Moreover, in that
case both definitions coincide.

By $\mathcal{Q}_m(\mathbf{f})$ we denote the monic polynomial whose
zeros are the poles of ${\bf f}$ in $D_{m}({\bf f})$ counting
multiplicities. The set of distinct zeros of
$\mathcal{Q}_m(\mathbf{f})$ is denoted by ${\mathcal{P}}_m({\bf
f})$.

\begin{defi}\label{defsimultaneos}
Let ${\bf f} = (f_1,\ldots,f_d)$ be a system of $d$ formal Taylor
expansions as in \eqref{sistema}. Fix a multi-index ${\bf m} =
({m_1},\ldots,m_d) \in {\mathbb{Z}}_+^d \setminus \{\bf 0\}$ where
${\bf 0}$ denotes the zero vector in ${\mathbb{Z}}_+^d$. Set $|{\bf
m}| = {m_1} +\cdots + m_d$. Then, for each $n \geq \max
\{{m_1},\ldots,m_d\}$, there exist polynomials $Q, P_k,
k=1,\ldots,d,$ such that
\begin{itemize}
\item[a.1)] $\deg P_k \leq n - m_k,\, k=1,\ldots,d,\quad \deg Q
\leq |\mathbf{m}|,\quad Q \not\equiv 0,$
\item[a.2)] $Q(z) f_k(z) - P_k (z) = A_k z^{n+1} + \cdots .$
\end{itemize}
The vector rational function ${\bf R}_{n,{\bf m}} = (P_1 /Q
,\ldots,P_d /Q)$ is called an $(n,{\bf m})$ Hermite-Pad\'e
approximation of ${\bf f}$.
\end{defi}
This vector rational approximation, in general, is not
uniquely determined   and in the sequel we assume that
given $(n,{\bf m})$ one particular solution is taken. For that
solution we write
\begin{equation} \label{incomplete}
{\bf R}_{n,{\bf m}} = (R_{n,{\bf m},1},\ldots,R_{n,{\bf m},d})=
(P_{n,{\bf m},1},\ldots,P_{n,{\bf m},d})/Q_{n,{\bf m}},
\end{equation}
where $Q_{n,{\bf m}}$ has no common zero simultaneously with all the
$P_{n,{\bf m},k}$ and is normalized to be monic unless otherwise
stated. Sequences $\{{\bf R}_{n,\bf m}\}$ for which $|\bf m|$ remains
fixed when $n$
 varies are called row sequences, and when
 $|{\bf m}| = {\mathcal{O}}(n),\, n \to \infty,$ diagonal sequences.

The study of simultaneous Hermite-Pad\'e approximations of systems
of functions has a long tradition (see \cite{Herm}) and they have
been subject to renewed interest in the recent past (see, for
instance, \cite{FL} and the references therein). Many papers deal
with diagonal sequences and their applications in different fields (number theory, random matrices, brownian motions, Toda lattices, to
name a few). At the same time, few papers study row sequences. In this
second direction a  significant contribution is due to
Graves-Morris/Saff in \cite{GS} where they prove an analogue of the
Montessus de Ballore theorem which plays a
central role in the classical theory of Pad\'e approximation. See also \cite{GS2}-\cite{GS3} for
different approaches to the same type of results as well as \cite{sidi} and references therein for least-squares versions.

Before going into details let us briefly describe the
scalar case ($d=1$) corresponding to classical Pad\'e
approximation which is well
understood. When $d=1$ we write $\textbf{f}=f,\,\textbf{m}=m\in
\mathbb{N},$ and ${\bf R}_{n,{\bf m}}=R_{n,m}$. Given a compact set
${K} \subset {\mathbb{C}}$, $\|\cdot\|_{K}$ denotes the sup norm on
$K$. We summarize what we need in the following statement.

\begin{A} \label{A}
 Let $f$ be a formal Taylor expansion about the origin and fix $m \in
\mathbb{N}$. Then, the following two assertions are equivalent.
\begin{itemize}
\item[a)] $R_0(f) > 0$ and $f$ has exactly $m$ poles in $D_m(f)$
counting multiplicities.
\item[b)] There is a polynomial $Q_m$ of degree $m,\, Q_m(0) \neq 0,$  such
that the sequence of denominators $\{Q_{n,m}\}_{ n \geq m}$ of the
Pad\'e approximations of $f$ satisfies
$$
\limsup_{n \to \infty} \|Q_m - Q_{n,m}\|^{1/n} = \theta < 1,
$$
where $\|\cdot\|$ denotes the coefficient norm in the space of
polynomials.
\end{itemize}
Moreover, if either a) or b) takes place then
$Q_m\equiv\mathcal{Q}_m(f)$,
\begin{equation} \label{eq:6a}
\theta=\frac{\max \{|\xi|: \xi \in {\mathcal{P}}_m(f)\}}{R_m(f)},
\end{equation}
and
\begin{equation}\label{eq:6}
\limsup_{n \to \infty}
\|f - R_{n,m}\|_{K}^{1/n} = \frac{\|z\|_ {K}}{R_{m}(f)},
\end{equation}
where ${K}$ is any compact subset of $D_m(f) \setminus
{\mathcal{P}}_m(f)$.
\end{A}
From this result it follows that if $\xi$ is a pole of $f$ in
$D_m(f)$ of order $\tau$, then for each $\varepsilon
>0$, there exists $n_0$ such that for $n \geq n_0$, $Q_{n,m}$ has
exactly $\tau$ zeros in $\{z: |z - \xi |< \varepsilon\}$. We say
that each pole of $f$ in $D_m(f)$ attracts as many zeros of
$Q_{n,m}$ as its order when $n$ tends to infinity.

So stated Gonchar's Theorem does not appear in the literature and
needs some comments. Under assumptions a), in \cite{Mon} Montessus
de Ballore proved that
\[ \lim_{n \to \infty}  Q_{n,m} = \mathcal{Q}_m(f),
\qquad \lim_{n\to \infty} R_{n,m} = f,
\]
with uniform convergence on compact subsets of $D_m(f) \setminus
{\mathcal{P}}_m(f)$ in the second limit. In essence, Montessus
proved that a) implies b) with $Q_m = \mathcal{Q}_m(f)$, showed that
$\theta \leq {\max \{|\xi|: \xi \in {\mathcal{P}}_m(f)\}}/{R_m(f)}$,
and proved \eqref{eq:6} with equality  replaced by $\leq$. These are
the so called direct statements of the theorem. The inverse
statements, b) implies a), $\theta \geq {\max \{|\xi|: \xi \in
{\mathcal{P}}_m(f)\}}/{R_m(f)}$, and the inequality $\geq$ in
\eqref{eq:6} are immediate consequences of \cite[Theorem 1]{gon2}.
The study of inverse problems of Pad\'e approximation was suggested
by A.A. Gonchar in \cite[Subsection 12]{gon2} where he presented
some interesting conjectures. Some of them were solved in
\cite{sue1} and \cite{sue2}. See \cite{apt} for a brief account of
Gonchar's most recent results and a list of his publications.

In \cite{GS}, Graves-Morris and Saff proved an analogue of the
direct part of Gonchar's Theorem for simultaneous approximation with
the aid of the concept of polewise independence of a system of
functions (for the definition, see \cite{GS}). They also established
upper bounds for the convergence rates corresponding to
\eqref{eq:6a} and \eqref{eq:6}. The Graves-Morris/Saff Theorem was
refined and complemented in \cite[Theorem 4.4]{cacoq} by weakening
the assumption of polewise independence, improving the upper bound
given in \cite{GS} for the rate \eqref{eq:6a}, and giving the exact
one for \eqref{eq:6}. Until now, results of inverse type for row
sequences of Hermite-Pad\'e approximants are not available.

Our purpose is to obtain an analogue of Gonchar's
Theorem for simultaneous Hermite-Pad\'e approximants, characterizing
the exact rates of convergence of the $Q_{n,\mathbf{m}}$ and
$\mathbf{R}_{n,\mathbf{m}}$.

The underlying idea in inverse-type results is that a polynomial
which is the limit of the denominators of the approximants must have as
zeros the poles of the function being approximated, provided
that the rate of convergence is geometric. However, the actual
situation in simultaneous approximation may be rather complicated as
the following example shows. Take ${\bf f} = (f_1,f_2)$, where
\begin{equation}\label{sistema1}
f_1 = \frac{1}{1-2z}+\sum_{n = 0}^\infty z^{n!} + \frac{1}{z-2},
\qquad f_2 = \frac{1}{1-2z}+\sum_{n = 0}^\infty z^{n!},
\end{equation}
and ${\bf m} = (1,1)$. It is clear that the unit circle is a natural
boundary of definition for both functions $f_1$ and $f_2$ and thus
$z=2$ cannot be a pole of ${\bf f}$ in any system of domains.
However, results contained in \cite{cacoq} show that the denominators
$Q_{n,\mathbf{m}}$ of the simultaneous Hermite-Pad\'e approximants
converge with geometric rate to the polynomial $(z-1/2)(z-2)$.

This kind of examples leads us to
introduce the following concept which is actually inspired by the
definition of polewise independence in \cite{GS}.

For each $r>0$, set $D_r=\{z\in\mathbb{C}\,:\,|z|<r\}$,
$\Gamma_r=\{z\in\mathbb{C}\,:\,|z|=r\}$, and
$\overline{D}_r=\{z\in\mathbb{C}\,:\,|z|\le r\}$.

\begin{defi} \label{e-pole}
Given ${\bf f} = (f_1,\ldots,f_d)$ and ${\bf m} = (m_1,\ldots,m_d) \in \mathbb{Z}^d \setminus \{{\bf 0}\}$ we say that $\xi\in\mathbb{C}\setminus\{0\}$ is a system pole of
order $\tau$ of  ${\bf f}$ with respect to $\bf m$ if
for each $s=1,\dots,\tau$ there exists at least one polynomial
combination of the form
\begin{equation}\label{combination}
\sum_{k=1}^d p_k f_k,\quad \deg p_k<m_k,\quad k=1,\dots,d,
\end{equation}
which is analytic on a neighborhood of $\overline{D}_{|\xi|}$ except
for a pole at $z=\xi$ of exact order $s$ and there is no
polynomial combination of the form \eqref{combination}  with those
properties for $s$ greater than $\tau$. If some component $m_k$
equals zero the corresponding polynomial $p_k$ is taken identically
equal to zero.
\end{defi}

The great advantage of this definition with respect to that of
polewise independence is that we have liberated it from establishing
a priori a region where the property should be verified. This turns
out to be crucial.

We wish to underline that if some component $m_k$ equals zero, that
component places no restriction on Definition \ref{defsimultaneos}
and does not report any benefit in finding system poles; therefore,
without loss of  generality we can restrict our attention to
multi-indices ${\bf m} \in \mathbb{N}^d$, and we will do so in the
sequel, except in reference to  the convergence of the approximants
themselves.

Notice that the definition of  system pole strongly depends on the
multi-index $\mathbf{m}$ and that a system $\mathbf{f}$ cannot have
more than $|\mathbf{m}|$ system poles with respect to $\mathbf{m}$
counting their order. During the proof of Theorem \ref{reciprocal}
below, carried out in Section \ref{simultaneos}, we give a procedure
for finding in a finite number of steps all the system poles of
$\mathbf{f}$ with respect to a multi-index $\mathbf{m}$ under
appropriate conditions.

It is easy to see that a system pole may not be a pole of ${\bf f}$
or viceversa. For example,
let ${\bf f}$ be the system given by \eqref{sistema1} and ${\bf m}
= (1,1)$. The point $z=2$, which lies beyond the natural
boundary of definition of  $f_1$ and $f_2$ is not a pole; however it is a system pole of $\mathbf{f}$ since $f_1 -f_2$ has a pole at $z=2$.

 On the other hand, take ${\bf f} = (f_1,f_2)$
with
\[ f_1 = \frac{1}{z-1} + \frac{1}{z-2} ,
\qquad f_2 =  \frac{1}{z-3},
\]
and ${\bf m} = (1,1)$. Then the points $z=1$ and $z=3$ are poles and
 system poles of ${\bf f}$ but $z=2$ is only a pole because there is no way
of eliminating the pole at $z=1$ through linear combinations of
$f_1$ and $f_2$ without eliminating the pole at $z=2$.

To each system pole $\xi$ of $\mathbf{f}$ with respect to
$\mathbf{m}$ we associate several characteristic values. Let $\tau$ be
the order
of $\xi$ as a system pole of $\mathbf{f}$. For each $s=1,\dots,\tau$
denote by $r_{\xi,s}(\mathbf{f},\mathbf{m})$ the largest of all the
numbers $R_s(g)$ (the radius of the largest disk containing at most
$s$ poles of $g$), where $g$ is a polynomial combination of type
\eqref{combination} that is analytic on a neighborhood of
$\overline{D}_{|\xi|}$ except for a pole at $z=\xi$ of order $s$.
Then
$$
\begin{array}{c}
\dst R_{\xi,s}(\mathbf{f},\mathbf{m})=\min_{k=1,\dots,s} r_{\xi,
{k}}(\mathbf{f},\mathbf{m}),
\\ \\ \dst
R_\xi(\mathbf{f},\mathbf{m})=
R_{\xi,\tau}(\mathbf{f},\mathbf{m})=\min_{s=1,\dots,\tau}
r_{\xi,s}(\mathbf{f},\mathbf{m}).
\end{array}
$$

Obviously, if $d=1$ and $({\bf f,\bf m}) = (f,m)$, system poles and
poles in $D_m(f)$ coincide. Also,
$R_\xi(\mathbf{f},\mathbf{m})=R_{m}(f)$ for each
 pole $\xi$ of $f$ in $D_{m}(f)$.

 By $\mathcal{Q}_{|\mathbf{m}|}(\mathbf{f},\mathbf{m})$ we denote the
monic polynomial whose zeros are the system poles of ${\bf f}$ with
respect to $\mathbf{m}$ taking account of their order. The set of
distinct zeros of
$\mathcal{Q}_{|\mathbf{m}|}(\mathbf{f},\mathbf{m})$ is denoted by
${\mathcal{P}}_{|\mathbf{m}|}({\mathbf f},\mathbf{m})$.

The following theorem constitutes our main result.

\begin{thm} \label{reciprocal} Let ${\bf f}$ be a system
of formal Taylor expansions as in \eqref{sistema} and fix a
multi-index $\mathbf{m}\in\mathbb{N}^d $. Then, the following two
assertions are equivalent.

\begin{itemize}
\item[a)]
$R_0(\mathbf{f}) > 0$ and ${\bf f}$ has exactly $|{\bf m}|$ system
poles with respect to $\mathbf{m}$ counting multiplicities.

\item[b)] The sequence of
denominators $\{Q_{n,{\bf m}}\}_{n \geq |{\bf m}|}$ of simultaneous
Pad\'e approximations of ${\bf f}$ is uniquely determined for all
sufficiently large $n$ and there exists a polynomial
$Q_{|\mathbf{m}|}$ of degree $|\mathbf{m}|,\, Q_{|\mathbf{m}|}(0) \neq
0,$ such that
$$
\limsup_{n \to \infty} \|Q_{|{\bf m}|} - Q_{n,{\bf m}}\|^{1/n} =
\theta < 1.
$$
\end{itemize}
Moreover, if either ${\rm a})$ or ${\rm b})$ takes place then
$Q_{|\mathbf{m}|}\equiv\mathcal{Q}_{|\mathbf{m}|}(\mathbf{f},\mathbf{m})$
and
\begin{equation} \label{eq:6b}
\theta=\max\left\{\frac {|\xi|}
{R_{\xi}(\mathbf{f},\mathbf{m})}\,:\, \xi \in
{\mathcal{P}}_{|\mathbf{m}|}(\mathbf{f},\mathbf{m}) \right\}.
\end{equation}
\end{thm}
If $d=1$,  $R_{n,m}$ and $Q_{n,m}$ are uniquely determined.
Therefore,  Theorem \ref{reciprocal} implies Gonchar's Theorem except
for \eqref{eq:6}
whose analogue will be presented in Section 3.2
to avoid introducing new notation at this stage.

The paper is structured as follows. In Section \ref{incompletos} we
continue with the study of incomplete Pad\'e approximants
initiated in \cite{cacoq} proving results of inverse type.
Section \ref{simultaneos} is dedicated to the proof of Theorem
\ref{reciprocal} and the analogue of \eqref{eq:6}.


\section{Incomplete Pad\'e approximants}\label{incompletos}

Let
\begin{equation} \label{i1} f(z) = \sum_{n=0}^{\infty} \phi_n z^n,
\qquad \phi_n \in {\mathbb{C}},
\end{equation}
denote a formal or convergent Taylor expansion about the origin.

\begin{defi}\label{defincompletos}
Let $f$  denote a formal Taylor expansion as in \eqref{i1}. Fix
$m\ge m^*\ge 1$. Let $n \geq m$. We say that the rational function
$r_{n,m}$ is an incomplete Pad\'e approximation of type $(n,m,{m^*})
$ corresponding to $f$ if $r_{n,m}$ is the quotient of any two
polynomials $p $ and $q $ that verify
\begin{itemize}
\item[b.1)]
$\deg p \le n-{m^*},\quad \deg q \le m,\quad q \not\equiv 0,$
\item[b.2)] $q(z) f(z)-p(z)=
A z^{n+1}+ \cdots .$
\end{itemize}
\end{defi}

Notice that given $(n,m,{m^*}), n \geq m \geq {m^*},$ any  of the
Pad\'e approximants $R_{n,{m^*}},\ldots,R_{n,m}$ can be regarded
an incomplete Pad\'e approximation of type $(n,m,{m^*})$ of $f$.
From Definition \ref{defsimultaneos} and
(\ref{incomplete}) it follows that $R_{n,{\bf m},k}, k=1,\ldots,d,$
is an incomplete Pad\'e approximation of type $(n,|\mathbf{m}|,m_k)$
with respect to $f_k$.

In the sequel, for each $n \geq m \geq {m^*},$ we choose one
candidate. After canceling out common factors between $q$ and $p$,
we write $ r_{n,m} = p_{n,m}/q_{n,m}, $ where, additionally,
$q_{n,m}$ is normalized to be monic. Suppose that $q$ and $p$ have a
common zero at $z=0$ of order $\lambda_n$. From b.1)-b.2) it follows
that
\begin{itemize}
\item[b.3)] $ \deg p_{n,m} \leq n-m^*-\lambda_n, \quad
\deg q_{n,m} \leq m-\lambda_n, \quad q_{n,m} \not\equiv 0,$
\item[b.4)] $q_{n,m}(z) f(z)-p_{n,m}(z)=
A z^{n+1-\lambda_n}+ \cdots .$
\end{itemize}
where $A$ is, in general, a different constant from the one in b.2).

The first difficulty encountered in dealing with inverse-type
results is to justify in terms of the data that the formal series
corresponds to an analytic element which does not reduce to a
polynomial. In our aid comes the next result, which provides such
information in terms of whether the zeros of the polynomials
$q_{n,m}$ remain away or not from $0$ and/or  $\infty$ as $n$ grows.
Let
\[ {\mathcal{P}}_{n,m} = \{\zeta_{n,1},\ldots,\zeta_{n,m_n}\},
\quad n\ge m, \quad m_n \leq m,
\]
denote the  collection of zeros of $q_{n,m}$ repeated according to
their multiplicity, where $\deg q_{n,m}=m_n$. Put
$$
S=\sup_{N\ge m}\inf\left\{|\zeta_{n,k}|: n \geq N, m_n \geq 1, 1
\leq k \leq m_n\right\}
$$
and
$$
G=\inf_{N\ge m}\sup\left\{|\zeta_{n,k}|: n \geq N, m_n \geq 1, 1
\leq k \leq m_n\right\}.
$$
Finally, set
$$
\tau_n = \min\{n-m^*-\lambda_n- \deg p_{n,m}, m -\lambda_n-
m_n\},\quad n\ge m.
$$
From b.3) we know that $\tau_n\ge 0,\, n\ge m$.
\begin{thm} \label{teo;6} Let $f$ be a formal power series as in \eqref{i1}.
Fix $m \geq m^* \geq 1$. The following assertions hold.
\begin{itemize}
\item[i)] If  $|\lambda_n - \lambda_{n-1}| \leq m^* -1, n \geq n_0,$
and $S>0$ then $R_0(f) > 0$.
\item[ii)] If $|(m_n + \lambda_n+\tau_n) -
(m_{n-1}+\lambda_{n-1}+ \tau_{n-1})| \leq m^* -1, n\geq n_0$, and
$G<\infty$ then either $f$ is a polynomial or $R_0(f) < \infty$. If,
additionally, there exists a sequence of indices $\Lambda$ such that
$\deg q_{n,m} \geq 1, n \in \Lambda,$ then $R_0(f) < \infty$.
\end{itemize}
\end{thm}

\begin{proof}
From definition
\begin{equation} \label{eq:1} (q_{n,m} f - p_{n,m})(z) = A z^{n+1-\lambda_n}
+ \cdots,
\end{equation}
and $q_{n,m}(0) \neq 0$.

We may suppose that $\inf\left\{|\zeta_{n,k}|: n \geq n_0, m_n \geq
1, 1 \leq k \leq m_n\right\}> 0$ and $|\lambda_n - \lambda_{n-1}|
\leq m^* -1,\, n\ge n_0$. Normalize $q_{n,m}$ as follows. If $m_n
\geq 1$ take
\[ q_{n,m}(z) = \prod_{k=1}^{m_n} \left( 1 - \frac{z}{\zeta_{n,k}} \right) =
a_{n,0}+ a_{n,1}z + \cdots + a_{n,m_n} z^{m_n},\quad a_{n,0} =1.
\]
Otherwise $q_{n,m}(z) \equiv 1 = a_{n,0}$.

Using the Vieta formulas connecting the coefficients of a polynomial
and its zeros it follows that there exists $C_1\ge 1$ such that
\begin{equation} \label{eq:3}
\sup \left\{|a_{n,k}| : 0\leq k\leq m_n,\, n \geq n_0\right\}\le
C_1< \infty.
\end{equation}
The coefficient corresponding to $z^k, k\in \{n-m^*-\lambda_n +1,
\ldots,n - \lambda_n\}$ in the left-hand side of
(\ref{eq:1}) equals
\begin{equation} \label{eq:2} \phi_k + a_{n,1} \phi_{k-1} + \cdots +
a_{n,m_n} \phi_{k-m_n} = 0,
\end{equation}
since $\deg p_{n,m} \leq n  -m^*- \lambda_n$.

If $m_n \geq 1$, \eqref{eq:3} and \eqref{eq:2} imply that
\[ |\phi_k| \leq C_1(|\phi_{k-1}| + \cdots + |\phi_{k-m_n}|).
\]
Therefore, for each $k \in \{n -m^* -\lambda_n +1,\ldots,n - \lambda_n\}$
there exists $k' \in
\{k-1,\ldots,k-m\}$ ($m_n \leq m$) such that
\begin{equation} \label{eq:a} |\phi_k|\leq C_1 m |\phi_{k'}|.
\end{equation}
Should $m_n = 0$, for the same values of $k$, we have $\phi_k = 0$
and \eqref{eq:a} is trivially verified. Substituting $n$ by $n-1$,
we deduce that for each $k \in \{n-m^*- \lambda_{n-1} ,\ldots,n -
\lambda_{n-1} -1\}$ there exists $k' \in \{k-1,\ldots,k-m\}$ such
that
\begin{equation} \label{eq:b} |\phi_k|\leq C_1 m |\phi_{k'}|.
\end{equation}

As $n \geq n_0$, we have $$ n-\lambda_{n-1}  \geq  n-\lambda_n
-m^*+1 $$ and $$ n-\lambda_{n-1} -m^* \leq n-\lambda_n - 1, $$
because $|\lambda_n - \lambda_{n-1}| \leq m^* -1$. Consequently, the
range of values taken by $k$ due to relations \eqref{eq:a} and
\eqref{eq:b} are either contiguous or overlapping for $n \geq n_0$.
Since $n- \lambda_n$ tends to $\infty$ as $n$ goes to $\infty$, we
conclude that for all $n \geq n_0$ there exists $n' \in
\{n-1,\ldots,n-m\}$ such that
\begin{equation} \label{eq:c} |\phi_n|\leq C_1 m |\phi_{n'}|.
\end{equation}

Let $\Lambda$ be a sequence of indices such that
\[ \lim_{n \in \Lambda} |\phi_n|^{1/n} = \limsup_{n \to \infty}
|\phi_n|^{1/n} = 1/R_0(f).
\]
Choose $n \in \Lambda$. Due to \eqref{eq:c} there exist indices
$n_1 > n_2> \cdots
> n_{r_n},\, n_{r_n}\leq n_0,$ where $r_n \leq n-n_0$,
such that
\[ |\phi_n| \leq C_1m|\phi_{n_1}| \leq \cdots \leq (C_1m)^{r_n}|
\phi_{n_{r_n}}|.
\]
Consequently,
\[1/R_0(f) = \lim_{n \in \Lambda} |\phi_n|^{1/n} \leq
\limsup_{n \to \infty} (C_1m)^{r_n/n} \leq C_1m.
\]
Therefore, $R_0(f) \geq (C_1m)^{-1} > 0$, which proves i).

As for ii), assume that  $\sup\{|\zeta_{n,k}|: n \geq n_0, m_n \geq
1, 1 \leq k \leq m_n\}< \infty$ and $|(m_n + \lambda_n +\tau_n) -
(m_{n-1}+ \lambda_{n-1}+\tau_{n-1})| \leq m^* -1,\, n\ge n_0$. Set
$t_n(z) = (z-1)^{\tau_n}$. Define $\tilde{q}_{n,m} = t_n q_{n,m}$
and $\tilde{p}_{n,m} = t_n p_{n,m}$. Normalize $\tilde{q}_{n,m}$ as
follows. If $m_n + \tau_n \geq 1$ take
\[ \tilde{q}_{n,m}(z) = \prod_{k=1}^{m_n + \tau_n}
\left( z-  {\zeta_{n,k}} \right) =
 b_{n,0}z^{m_n+\tau_n} + \cdots + b_{n,m_n +\tau_n-1}z +  b_{n,m_n +\tau_n},
\]
where $b_{n,0} =1.$ Should $m_n + \tau_n =0$ we set $\tilde{q}_{n,m}
\equiv 1 = b_{n,0}$. Using the Vieta formulas, it follows that there
exists $C_2\ge 1$ such that
\begin{equation} \label{eq:4} \sup \left\{|b_{n,k}| :
0\leq k\leq m_n,\, n \geq n_0\right\}\le
C_2< \infty.
\end{equation}
The coefficient corresponding to $z^k, k\in \{n-m^*-\lambda_n +1,
\ldots,n - \lambda_n\}$, in the left-hand side of
(\ref{eq:1}) equals
\begin{equation} \label{eq:5}  \phi_{k -m_n -\tau_n} + b_{n,1}
\phi_{k -m_n -\tau_n+1} + \cdots + b_{n,m_n +\tau_n} \phi_{k}  = 0,
\end{equation}
since $\deg \tilde{p}_{n,m} \leq n-m^*-\lambda_n $.

Should $m_n + \tau_n \geq 1$, \eqref{eq:4} and \eqref{eq:5} imply that
\[ |\phi_{k-m_n -\tau_n}| \leq C_2(|\phi_{k-m_n -\tau_n+1}| + \cdots +
|\phi_{k}|),
\]
or, what is the same, for each $k\in \{n-m^*-\lambda_n - m_n -\tau_n
+1, \ldots,n - \lambda_n - m_n-\tau_n\}$, we have
\[ |\phi_{k }| \leq C_2(|\phi_{k  +1}| + \cdots + |\phi_{k + m_n+\tau_n}|).
\]
Therefore, for  each $k\in \{n-m^* -\lambda_n - m_n -\tau_n+1,\ldots,n -
\lambda_n - m_n-\tau_n\}$ there exists $k' \in
\{k+1,\ldots,k+m\}$ $(m_n + \tau_n \leq m)$ such that
\begin{equation} \label{eq:7} |\phi_{k'}| \geq \frac{|\phi_{k}|}{C_2m}.
\end{equation}
In case that $m_n + \tau_n = 0$ we have
 $\phi_k = 0$ for the same values of $k$ and \eqref{eq:7} is also true.

Using the assumption that $|\lambda_n + m_n +\tau_n- \lambda_{n-1} -
m_{n-1}-\tau_{n-1}| \leq m^* -1$, it is easy to check, similarly to
the previous case, that the range of values taken by the parameter
$k$ for consecutive values of $n$ are either contiguous or
overlapping. Also, $n - \lambda_n - m_n -\tau_n$ tends to $\infty$
as $n$ goes to $\infty$. Consequently, from \eqref{eq:7} we have
that for all $n \geq n_0$ there exists $n' \in \{n+1,\ldots,n+m\}$
such that
\begin{equation} \label{eq:d}
|\phi_{n'}| \geq \frac{|\phi_{n}|}{C_2m}
\end{equation}

Using \eqref{eq:d} we can find an increasing sequence of multi-indices
$\{n_{s}\}_{s \in {\mathbb{Z}}_+}$, $n_{s+1} \in \{n_{s}+1,
\ldots,n_{s} +m\} $ and $n_1 \in \{n_0,\ldots,n_0 +
m\}$ such that
\[ |\phi_{n_{s+1}}| \geq \frac{|\phi_{n_1}|}{(C_2m)^{s}}.
\]
Should $f$ be a polynomial there is nothing to prove. Otherwise,
changing the value of $n_0$ if necessary, without loss of generality
we can assume that $\phi_{n_1} \neq 0$. Then,
\[\liminf_{s\to \infty} |\phi_{n_{s+1}}|^{1/n_{s+1}} \geq
\frac{1}{\dst\limsup_{s \to \infty} (C_2m)^{s/n_{s+1}}} \geq
\frac{1}{C_2m},
\]
since
\[ \limsup_{s \to \infty} \frac{s}{n_{s+1}} \leq
\limsup_{s \to \infty} \frac{s}{n_1 + s} = 1.
\]
It follows that
\[  R_0(f) = \frac{1}{\dst\limsup_{n \to \infty}|\phi_n|^{1/n}}
\leq \frac{1}{\dst\liminf_{s \to \infty}|\phi_{n_{s+1}}|^{1/n_{s+1}} }
\leq C_2m < \infty,
\]
as we needed to prove.

Finally, if $f$ is a polynomial,  say of degree $N$, we would have
that for all $n \geq N+m,$ $f \equiv p_{n,m}/q_{n,m}$ and $q_{n,m}
\equiv 1$. Consequently, if there exists $\Lambda$ such that $\deg
q_{n,m} \geq 1, n \in \Lambda,$ $f$ cannot be a polynomial and,
therefore, only $R_0(f) < \infty$ is possible.
\end{proof}

\begin{lem} \label{lem:sufi} A sufficient condition to have
$| \lambda_n - \lambda_{n-1}| \leq m^* -1$ and $|(m_n + \lambda_n+\tau_n) -
(m_{n-1}+\lambda_{n-1}+\tau_{n-1})| \leq m^* -1$   is that
$$
\min\left\{m_n + \tau_n,\, m_{n-1} + \tau_{n-1}\right\} \geq m - m^*
+1.
$$
\end{lem}

\begin{proof} In fact, for
$k=n-1$ and $k=n$, if $m_k +\tau_k \geq m-m^* +1$ then $0 \leq
\lambda_k \leq m^*-1$ because $\lambda_k + m_k +\tau_k\leq m$  and
the first inequality readily follows. On the other hand,
$$
\begin{array}{c}
|(m_n + \lambda_n+\tau_n) - (m_{n-1}+\lambda_{n-1}+\tau_{n-1})|  \\
\\ = |(m_n + \lambda_n +\tau_n- m + m^* -1) - (m_{n-1}+\lambda_{n-1}+
\tau_{n-1}- m + m^* -1)|
\end{array}
$$
and $0 \leq  m_k +\lambda_k +\tau_k- m + m^* -1\leq
m^* -1$ for $k=n-1$ and $k=n$. Therefore, the second inequality also
holds.
\end{proof}

Applied to Pad\'e approximation $(m^* = m)$, Theorem \ref{teo;6} and
Lemma \ref{lem:sufi} imply that if $\deg Q_{n,m} \geq 1$
and its zeros remain uniformly bounded away
from $0$ and $\infty$, for sufficiently large $n$,
then $0 < R_0(f) < \infty$. This result has
not been stated elsewhere.

Let us see some consequences of Theorem \ref{teo;6} and Lemma
\ref{lem:sufi} on the extendability of a formal power series
and the location of some of its poles in terms of the behavior
of the zeros of the approximants. First we bring your attention
to some results from \cite{cacoq}.

Let $B$ be a subset of the complex plane $\mathbb{C}$. By
$\mathcal{U}(B)$ we denote the class of all coverings of $B$ by at
most a numerable set of disks. Set
$$
\sigma(B)=\inf\left\{\sum_{i=1}^\infty
|U_i|\,:\,\{U_i\}\in\mathcal{U}(B)\right\},
$$
where $|U_i|$ stands for the radius of the disk $U_i$. The quantity
$\sigma(B)$ is called the $1$-dimensional Hausdorff content of the
set $B$.

Let $\{\varphi_n\}_{n\in\mathbb{N}}$ be a sequence of functions
defined on a domain $D\subset\mathbb{C}$ and $\varphi$ another
function defined on $D$. We say that
$\{\varphi_n\}_{n\in\mathbb{N}}$ converges in $\sigma$-content to
the function $\varphi$ on compact subsets of $D$ if for each compact
subset $K$ of $D$ and for each $\varepsilon
>0$, we have
$$
\lim_{n\to\infty} \sigma\{z\in K :
|\varphi_n(z)-\varphi(z)|>\varepsilon\}=0.
$$
We denote this writing $\sigma$-$\lim_{n\to \infty} \varphi_n =
\varphi$ inside $D$.

We define the number $R^*_m(f)$ as the radius of the largest disk
centered at the origin on compact subsets of which the sequence
$\{r_{n,m}\}_{n \geq m}$ converges to $f$ in $\sigma$-content. In
\cite{cacoq} we gave a formula to produce this number and showed
that it depends on the specific sequence of incomplete Pad\'e
approximants considered. Set $D^*_m(f)=\left\{z\in\mathbb{C} :
|z|<R^*_m(f)\right\}$.

 Among other direct-type results, we proved that
\begin{equation}\label{incluido}
R_{m^*}(f)\le R^*_m(f)\le R_{m}(f),
\end{equation}
that $R_m^*(f) >0$ implies $R_0(f) > 0$, and that each pole of the
function $f$ in $D^*_{m}(f)$
attracts, with geometric rate, at least as many zeros of $q_{n,m}$
as its order (see \cite[Theorem 3.5]{cacoq}). Therefore, Theorem \ref{teo;6} and Lemma
\ref{lem:sufi} imply

\begin{cor} \label{cor;1} Let $f$ be a formal power series
as in \eqref{i1}. Fix $m \geq m^* \geq 1$. Assume that there exists
a polynomial ${q}_m$ of degree greater than or equal to
 $ m-m^*+1,\,{q}_m(0) \neq 0,$ such that $\lim_{n
\to \infty} {q}_{n,m} = {q}_m$. Then $0<R_0(f)<\infty$ and the zeros
of ${q}_{m}$ contain all the poles, counting multiplicities, that
$f$ has in $D^*_{m}(f)$.
\end{cor}

We need a relaxed version of Corollary \ref{cor;1}  for
the proof of Theorem \ref{reciprocal}.

\begin{lem} \label{rmayor0} Let $f$ be a formal power series
as in \eqref{i1} that is not a polynomial. Fix $m \geq m^* \geq 1$.
Let $r_{n,m}=\tilde{p}_{n,m}/\tilde{q}_{n,m}$ be an incomplete
Pad\'e approximant of type $(n,m,m^*)$ corresponding to $f$, where
$\tilde{p}_{n,m}$ and $\tilde{q}_{n,m}$ are obtained from Definition
\ref{defincompletos} and common factors between them are allowed.
Assume that there exists a polynomial $\tilde{q}_m$ of degree
 $ m,\, \tilde{q}_m(0) \neq 0,$ such that $\lim_{n
\to \infty} \tilde{q}_{n,m} = \tilde{q}_m$. Then $0<R_0(f)<\infty$
and the zeros of $\tilde{q}_{m}$ contain all the poles, counting
multiplicities, that $f$ has in $D^*_{m}(f)$.
\end{lem}

\begin{proof}
Let us show that the assumptions of Lemma \ref{lem:sufi} are
verified for the incomplete approximant $r_{n,m}$. Let
$r_{n,m}={p}_{n,m}/{q}_{n,m}$, where the polynomials $p_{n,m}$ and
$q_{n,m}$ are relatively prime. Since $\tilde{q}_m(0) \neq 0$, then
$\tilde{q}_{n,m}(0) \neq 0,\, n \geq n_0$. Thus, $\tilde{p}_{n,m}$
and $\tilde{q}_{n,m}$ do not have a common zero at $z=0$ and
$\lambda_n=0$ for all $n \geq n_0$. As before, set $m_{n} = \deg
q_{n,m}$ and
\[ \tau_{n} = \min\left\{n-m^* - \deg p_{n,m},\, m
- m_{n}\right\}, \qquad n \ge n_0.
\]
Notice that $\tau_{n}= m-m_n,\, n\ge n_0,$ because the polynomials
$q_{n,m}$ and $p_{n,m}$ are obtained eliminating possible common
factors between $\tilde{q}_{n,m}$ and $\tilde{p}_{n,m}$ and by
assumption
\[ \min\left\{n-m^* - \deg \tilde{p}_{n,m},\, m
- \deg \tilde{q}_{n,m}\right\} = 0, \qquad n \geq n_0.
\]
Therefore, we have
\[ m_{n} + \tau_{n} =m \geq m - m^* +1, \qquad n
\geq n_0,
\]
and   Lemma \ref{lem:sufi} is applicable.

From  Theorem \ref{teo;6} we obtain $0<R_0(f)<\infty$. Now, from the
fact that each pole of $f$ in $D^*_m(f)$ attracts as many zeros of
$q_{n,m}$ as its order it follows that the zeros of $\tilde{q}_m$
contain all the poles, counting multiplicities, that $f$ has in
$D^*_m(f)$.
\end{proof}

In case that there exists $R > R_{m^*}(f)$ inside of which $f$ is
meromorphic then $D_R$ contains at least $m^* +1$ poles of $f$ since
$D_{m^*}(f)$ is the largest disk where $f$ is meromorphic with at
most $m^*$ poles. We can prove the following inverse-type result.

\begin{thm}\label{inverso}
Fix $m \geq m^* \geq 1$. Let $f$ be a formal power series as in
\eqref{i1} that is not a rational function with at most $m^*-1$
poles. Let $r_{n,m}=\tilde{p}_{n,m}/\tilde{q}_{n,m}$ be an
incomplete Pad\'e approximant of type $(n,m,m^*)$ corresponding to
$f$, where $\tilde{p}_{n,m}$ and $\tilde{q}_{n,m}$ are obtained from
Definition \ref{defincompletos} and common factors between them are
allowed. Suppose that there exists a polynomial $\tilde{q}_m,$ of
degree $m,\, \tilde{q}_m(0) \neq 0,$ such that
\begin{equation}\label{coro}
\limsup_{n\to\infty}\|\tilde{q}_{n,m}-\tilde{q}_m\|^{1/n}=\theta<1.
\end{equation}
Then, either $f$ has exactly $m^*$ poles in $D_{m^*}(f)$, which are
zeros of $\tilde{q}_m$ counting multiplicities, or $R_0(\tilde{q}_m
f)> R_{m^*}(f)$.
\end{thm}
\begin{proof}
From Lemma  \ref{rmayor0} we have $R_0(f) > 0$. So, $f$ is analytic
in a neighborhood of $z=0$. We also know that $R_0(\tilde{q}_m f)
\geq R_{m^*}(f)$ since the zeros of $\tilde{q}_m$ contain all the
poles that $f$ has in $D_{m^*}(f)$. Assume that $R_0(\tilde{q}_m f)=
R_{m^*}(f)$. Let us show that then $f$ has exactly $m^*$ poles in
$D_{m^*}(f)$. To the contrary, suppose that $f$  has in $D_{m^*}(f)$
at most $m^* -1$ poles. Then there exists a polynomial $q_{m^*}$,
with $\deg q_{m^*} <{m^*}$, such that
\[ R_0(q_{m^* } f) = R_{m^*}(f) =
R_0(q_{m^* }\tilde{q}_m f).
\]

Let
\[ q_{m^* }(z)\,\tilde{q}_m(z)\, f(z) = \sum_{n=0}^{\infty} a_n
z^n,
\]
then
\[R_{m^*}(f) =
R_0(q_{m^* }\tilde{q}_m f) = 1/ \limsup_{n\to\infty}
\sqrt[n]{|a_{n}|}.
\]
The $n$-th Taylor coefficient of $q_{m^*}[\tilde{q}_{n,m} f -
\tilde{p}_{n,m}]$ is equal to zero. Therefore, the $n$-th Taylor
coefficients of $q_{m^* }\tilde{q}_m f$ and $q_{m^* }\tilde{q}_m f -
q_{m^*}\tilde{q}_{n,m} f + q_{m^*}p_{n,m}$ coincide. Take $0 < r <
R_{m^*}(f)$ and recall that $\Gamma_r=\{z\in\mathbb{C}\,:\,|z|=r\}$.
Hence
$$
\begin{array}{rcl}\dst a_{n}& =&\dst \frac{1}{2\pi i} \int_{\Gamma_r}
\frac{[q_{m^* }\tilde{q}_m f - q_{m^*}\tilde{q}_{n,m} f +
q_{m^*}p_{n,m}](\omega)}{\omega^{n+1}}\, d\omega\\ \\ & =& \dst
 \frac{1}{2\pi i} \int_{\Gamma_r} \frac{[\tilde{q}_m
  - \tilde{q}_{n,m} ](\omega)\,
q_{m^*}(\omega) f(\omega)}{\omega^{n+1}}\, d\omega.
\end{array}
$$
Making use of (\ref{coro}) it readily follows that
\[
\frac{1}{R_{m^*}(f) }=\limsup_{n\to\infty}\sqrt[n]{|a_{n}|} \leq
\frac{\theta}{r}.
\]
Letting $r$ tend to $R_{m^*}(f)$ we have
\[\frac{1}{R_{m^*}(f) } \leq \frac{\theta}{R_{m^*}(f)}, \qquad \theta < 1,
\]
which implies that $R_{m^*}(f) = \infty$. Let us show that this is
not possible.

In fact,
\[
\left[q_{m^*}\,\tilde{q}_{n,m}\,f - q_{m^*}\,
\tilde{p}_{n,m}\right](z) = A_n z^{n+1} + \cdots,
\]
and $\deg q_{m^*}\tilde{p}_{n,m} \leq n  -1$. It follows that
$(q_{m^*} \tilde{p}_{n,m})/\tilde{q}_{n,m} = (q_{m^*}p_{n,m})/q_{n,m}$ is an
incomplete Pad\'e approximant of the function $q_{m^*} f$ of type
$(n,m,1)$, where the polynomials $p_{n,m}$ and $q_{n,m}$ are
relatively prime. As $\tilde{q}_{n,m}(0)\not =0,\, n\ge n_0$, the
polynomials $q_{m^*} \tilde{p}_{n,m}$ and $\tilde{q}_{n,m}$ do not
have a common zero at $z=0$ and $\lambda_n=0$ for all $n\ge n_0$.
Again, set $m_{n} = \deg q_{n,m}$ and
\[
{\tau}_{n} = \min\left\{n-1 - \deg p_{n,m},\,
 m  - m_{n}\right\}.
\]
Notice that ${\tau}_{n}= m -m_{n},\, n\ge n_0,$ because
\[ \min\left\{n-1 - \deg q_{m^*}\tilde{p}_{n,m},\,
 m  - \deg \tilde{q}_{n,m}\right\}= 0,\quad n\ge n_0.
\]
Thus, $m_{n} + {\tau}_{n} = m,\, n\ge n_0$. Using Lemma
\ref{lem:sufi} (for $m^* =1$) and Theorem \ref{teo;6} we conclude
that either $R_0(q_{m^*} f) < \infty$ or $q_{m^*} f$ is a
polynomial. However, the latter is not possible by hypotheses. On
the other hand, $R_0(q_{m^*} f) < \infty$ contradicts $R_{m^*}(f) =
\infty$. As claimed, $f$ has exactly $m^*$ poles in $D_{m^*}(f)$.
\end{proof}


\section{Simultaneous approximation}\label{simultaneos}
Throughout this section, $\mathbf{f}=(f_1,\dots,f_d)$ denotes a
system of formal power expansions as in \eqref{sistema} and
$\mathbf{m}=(m_1,\dots,m_d)\in\mathbb{N}^d$ is a fixed multi-index.
We are concerned with the simultaneous approximation of $\textbf{f}$
by sequences of vector rational functions defined according to
Definition \ref{defsimultaneos} taking account of
(\ref{incomplete}). That is, for each $n\in\mathbb{N}, n \geq |{\bf
m}|$, let $\left(R_{n,\mathbf{m},1},\dots,R_{n,\mathbf{m},d}\right)$
be a Hermite-Pad\'e approximation of type $(n,\mathbf{m})$
corresponding to $\mathbf{f}$.

As we mentioned earlier, $R_{n,\mathbf{m},k}$ is an incomplete
Pad\'e approximant of type $(n,|\mathbf{m}|,m_k)$ with respect to
$f_k,\,k=1,\dots,d$. Thus, from \eqref{incluido} we have
$$
D_{m_k}(f_k)\subset D^*_{|\mathbf{m}|}(f_k) \subset
D_{|\mathbf{m}|}(f_k),\quad k=1,\dots,d.
$$

\begin{defi} A vector ${\bf f} = (f_1,\ldots,f_d)$ of formal power
expansions is said to be algebraically independent with respect to
${\bf m} = (m_1,\ldots,m_d) \in {\mathbb{N}}^d$ if there do not
exist polynomials $p_1,\ldots,p_d$, at least one of which is
non-null, such that
\begin{itemize}
\item[c.1)] $\deg p_k \leq m_k -1,\, k=1,\ldots,d$,
\item[c.2)] $\sum_{k=1}^d p_k f_k$ is a polynomial.
\end{itemize}
\end{defi}
In particular, algebraic independence implies that for each
$k=1,\ldots,d$, $f_k$ is not a rational function with at most $m_k-1$ poles.
Notice that algebraic independence
may be verified solely in terms of the coefficients of the formal
Taylor expansions defining the system $\mathbf{f}$.

Given  $\mathbf{f}=(f_1,\dots,f_d)$  and
$\mathbf{m}=(m_1,\dots,m_d)\in {\mathbb{N}}^d$, we consider the
associated system $\overline{\mathbf{{f}}}$  of formal power
expansions
\[
\overline{\mathbf{{f}}}=
(f_1,\ldots,z^{m_1-1}f_1,f_2,\ldots,z^{m_d-1}f_d) =
(\bar{f}_{1},\ldots,\bar{f}_{|\bf m|}).
\]
We also define an associated multi-index $\overline{\mathbf{m}}$
given by $ \overline{\mathbf{m}}=(1,1,\dots,1) $ with
$|\overline{\mathbf{m}}|=|\mathbf{m}|$. The systems $\bf f$ and
$\overline{\bf f}$  share most properties. In particular, poles of
$\mathbf{f}$ and $\overline{\mathbf{f}}$ coincide and
$R_m(\mathbf{f})=R _m(\overline{\mathbf{f}}),\, m\in \mathbb{Z}_+$.

From the definition it readily follows that ${\bf f}$ is
algebraically independent with respect to $\bf m$ if and only if
there do not exist constants $c_k, k=1,\ldots,|{\bf m}|$, not all
zero, such that
\[ \sum_{k=1}^{|{\bf m}|}  c_k \bar{f}_k
\]
is a polynomial. That is, ${\bf f}$ is algebraically independent
with respect to $\bf m$ if and only if $\overline{\mathbf{f}}$ is
algebraically independent with respect to $\overline{\bf m}$. By the
same token, the system poles of ${\bf f}$ with respect to $\bf m$
are the same as the system poles of $\overline{\mathbf{f}}$ with
respect to $\overline{\bf m}$.

Finally, it is very easy to check that, for all $n\ge |\mathbf{m}|$,
the equations that define the common denominator $Q_{n,\mathbf{m}}$
for  $(\mathbf{f},\bf m)$ are the same as those defining
$Q_{n,\overline{\mathbf{m}}}$ for
$(\overline{\mathbf{f}},\overline{\mathbf{m}})$ and, consequently,
both classes of polynomials coincide.

\begin{lem} \label{lemuni} Let ${\bf f} = (f_1,\ldots,f_d)$ be a system
of formal Taylor expansions as in \eqref{sistema} and fix a
multi-index $\mathbf{m}\in {\mathbb{N}}^d$. Suppose that for all
$n\geq n_0$ the polynomial $ Q_{n,{\bf m}} $ is unique and $\deg
Q_{n,{\bf m}} =|{\bf m}|$. Then, the system ${\bf f}$ is
algebraically independent with respect to ${\bf m}$.
\end{lem}
\begin{proof}
Because of what said just before the statement of Lemma \ref{lemuni},
we can assume without loss of generality that
$\mathbf{m}=(1,1,\dots,1)$ and $d = |\bf m|$. We argue by contradiction.
Suppose that
there exist constants $c_k,\, k=1,\dots,d,$ not all zero, such that
$\sum_{k=1}^{d} c_k f_k$ is a polynomial. Should $d=1$,
$Q_{n,\mathbf{m}}\equiv 1$ for all $n$ sufficiently large and $\deg
Q_{n,\mathbf{m}}<1=|\mathbf{m}|$. If $d>1$, without loss of
generality, we can assume that $c_1\not=0$. Then
$$
f_1=p-\sum_{k=2}^{d} c_k f_k,
$$
where $p$ is a polynomial, say of degree $N$.

On the other hand, for each $n\ge d-1$, there exist polynomials
$Q_n, P_{n,k}$, $k=2,\ldots,d,$ such that
\begin{enumerate}
\item[-] $\deg P_{n,k} \leq n - 1,\, k=2,\ldots,d,\quad \deg Q_n
\le d-1,\quad Q_n \not\equiv 0,$
\item[-] $Q_n(z)\, f_k(z) - P_{n,k} (z) = A_k z^{n+1} + \cdots,
\quad k=2,\ldots,d.$
\end{enumerate}
Therefore,
\[
Q_n(z)\left( p(z)-\sum_{k=2}^d c_kf_k(z)\right)-\left( Q_n(z)\,
p(z)-\sum_{k=2}^d c_k P_{n,k}(z)\right) = A z^{n+1}+\dots
\]
and, for $n\ge d+N$, the polynomial $P_{n,1}=Q_n\, p-\sum_{k=2}^d
c_k P_{n,k}$ verifies $\deg P_{n,1}\le n-1$. Thus, for all $n$
sufficiently large, the polynomials $P_{n,k},\, k=1,\dots,d,$
satisfy Definition \ref{defsimultaneos} with respect to $\mathbf{f}$
and $\mathbf{m}$. Naturally, $Q_n$ gives rise to a polynomial
$Q_{n,\mathbf{m}}$ with $\deg Q_{n,\mathbf{m}}<d=|\mathbf{m}|$
against our assumption on $Q_{n,\mathbf{m}}$.
\end{proof}

Set
$$
\mathbf{D}_{\mathbf{m}}^*(\mathbf{f})=\left(D^*_{|\mathbf{m}|}(f_1),\dots,
D^*_{|\mathbf{m}|}(f_d) \right).
$$

The following corollaries are straightforward consequences of
Corollary \ref{cor;1} and Theorem \ref{inverso}, respectively,
together with the fact that, for each $k=1,\dots,d$,
$R_{n,\mathbf{m},k}=P_{n,\mathbf{m},k}/Q_{n,\mathbf{m}}$ is an
incomplete Pad\'e approximant of type $(n,|\mathbf{m}|,m_k)$ with
respect to $f_k$.
\begin{cor} \label{cor;2a} Let ${\bf f} = (f_1,\ldots,f_d)$ be a system
of formal Taylor expansions as in \eqref{sistema} and fix a
multi-index $\mathbf{m}\in\mathbb{N}^d$. Assume that ${\bf f}$ is
algebraically independent with respect to ${\bf m}$ and there exists
a polynomial $Q_{|{\mathbf{m}}|}$ of degree $ |\mathbf{m}|,\,
Q_{|{\mathbf{m}}|}(0) \neq 0, $ such that $\lim_{n \to \infty}
Q_{n,\mathbf{m}} = Q_{|{\mathbf{m}}|}$. Then $R_0(\mathbf{f})>0$,
the zeros of $Q_{|{\mathbf{m}}|}$ contain all the poles that
$\mathbf{f}$ has in $\mathbf{D}^*_{\mathbf{m}}(\mathbf{f})$, and
$R_0(f_k) < \infty$ for each $k=1,\ldots,d$.
\end{cor}

\begin{cor} \label{cor;4} Let ${\bf f} = (f_1,\ldots,f_d)$ be a system
of formal Taylor expansions as in \eqref{sistema} and fix a
multi-index $\mathbf{m}=(m_1,\dots,m_d)\in\mathbb{N}^d$. Assume that
${\bf f}$ is algebraically independent with respect to ${\bf m}$ and
there exists a polynomial $Q_{|\mathbf{m}|}$ of degree
$|\mathbf{m}|,\, Q_{|\mathbf{m}|}(0) \neq 0,$ such that
$$
\limsup_{n \to \infty} \|Q_{|{\bf m}|} - Q_{n,{\bf m}}\|^{1/n} =
\theta < 1.
$$
Then, for each $k=1,\ldots,d,$  either $f_k$ has exactly $m_k$ poles
in $D_{m_k}(f_k)$ or $R_0(Q_{|{\bf m}|}f_k) > R_{m_k}(f_k)$.
\end{cor}

\subsection{Proof of Theorem \ref{reciprocal}}\label{proof}

Let us prove first that b) implies a). From Lemma \ref{lemuni} it
follows that $\mathbf{f}$ is algebraically independent with respect
to $\mathbf{m}$ and, in turn, from Corollary \ref{cor;2a} we know
that $R_0({\bf f})
>0$. So, it is enough to prove that ${\bf f}$ has exactly $|{\bf m}|$
system poles with respect to $\mathbf{m}$ and without loss of
generality we can assume that $\mathbf{m}=(1,1,\dots,1)$.

We divide the proof into two parts. First, we collect a set of
$|{\bf m}|$ candidates to be system poles of $\mathbf{f}$  and prove
that they are the zeros of
$Q_{|\mathbf{m}|}$. We also prove that any system pole of
$\mathbf{f}$ must be among these candidates. In the second part
we prove that all these points previously collected are actually
system poles of $\mathbf{f}$.

In the disk $D_0({\bf f})$ there cannot be system poles of
$\mathbf{f}$ since all the functions $f_k$ are analytic. Now, for
each $k=1,\ldots,d$, by Corollaries \ref{cor;4} and \ref{cor;2a},
either the disk $D_1(f_k)$ contains exactly one pole of $f_k$, and
it is a zero of $Q_{|{\bf m}|}$, or $R_0(Q_{|{\bf m}|}f_k)
> R_1(f_k)$. Therefore, $D_0({\bf f})\not=\mathbb{C}$ and
$Q_{|{\bf m}|}$ contains as zeros all the poles of $f_k$ on the
boundary of $D_0(f_k)$ counting their order for
$k=1,\ldots,d=|\mathbf{m}|$. Moreover, the functions $f_k$ cannot
have on the boundary of $D_0( f_k)$ singularities other than poles.

According to this, the poles  of ${\bf f}$ on the boundary
of $D_0({\bf f})$ are all zeros of $Q_{|{\bf m}|}$ counting
multiplicities and the boundary contains no other singularity except poles.
Let us call them candidate system poles of
$\mathbf{f}$ and denote them by $a_1,\ldots,a_{n_1}$ taking account
of their order. Obviously, any system pole of $\mathbf{f}$ on the boundary
of $D_0({\bf
f})$ must  be one of the candidates since no linear combination of the
functions in $\bf f$ can produce poles at any other point.

Since $\deg Q_{|{\bf m}|} = |{\bf m}|$ we have $n_1 \leq |{\bf m}|$.
Should $n_1 = |{\bf m}|$ we have found all the candidates we were
looking for. Let us assume that $n_1 < |{\bf m}|$. We can find
coefficients $c_1,\ldots,c_{|{\bf m}|}$ such that $$
\sum_{k=1}^{|{\bf m}|} c_k f_k
$$
is analytic in a neighborhood of $\overline{D_0(\mathbf{f})}$.
Finding the coefficients $c_k$ reduces to solving a linear
homogeneous system of $n_1$ equations with $|{\bf m}|$ unknowns. In
fact, if $z=a$ is a candidate  system pole of $\mathbf{f}$
with multiplicity $\tau$ we obtain $\tau$ equations choosing the
coefficients $c_k$ so that
\begin{equation} \label{ecuacion} \int_{|\omega - a| =
\delta} (\omega - a)^i \left(\sum_{k=1}^{|{\bf m}|} c_k
f_k(\omega)\right) d\omega = 0, \qquad i=0,\ldots,\tau -1.
\end{equation}
where $\delta$ is sufficiently small.  We do the same with each
distinct candidate on the boundary of $D_0({\bf f})$. The linear homogeneous
system of equations so
obtained has at least $|{\bf m}|-n_1$ linearly independent solutions
which we denote by ${\bf c}^1_j,$ $ j=1,\ldots,|{\bf m}| -
n_1^*,\,n_1^*\le n_1.$

Set
\[ {\bf c}^1_j = (c_{j,1}^1,\ldots,c_{j,|{\bf m}|}^1),
\quad j=1,\dots,|{\bf m}| - n_1^*.
\]
Construct the $(|{\bf m}|-n_1^*) \times |{\bf m}|$ dimensional
matrix
\[ C^1 = \left(
\begin{array}{c}
{\bf c}^1_1 \\
\vdots \\
{\bf c}^1_{|{\bf m}|-n_1^*}
\end{array}
\right).
\]
Define the system $\mathbf{g}_1$ of $|{\bf m}| -n_1^*$ functions by
means of
\[ {\bf g}_1^t= C^1 {\bf f}^t = (g_{1,1},\ldots,g_{1,|{\bf m}|-n_1^*})^t,
\]
where $(\cdot)^t$ means taking transpose. We have
\[ g_{1,j} = \sum_{k=1}^{|{\bf m}|} c_{j,k}^1 f_{k},\qquad
j=1,\dots,|{\bf m}|-n_1^*.
\]

As the rows of $C^1$ are non-null, none of the functions $g_{1,j}$ are
polynomials because of the algebraic independence of
$\mathbf{f}$ with respect to $\mathbf{m}=(1,1,\dots,1)$.

Consider the region
$$
 D_0({\bf g}_1) = \bigcap_{j=1}^{|{\bf m}|-n_1^*} D_0(g_{1,j}).
$$
Obviously, by construction,  $ D_0({\bf f})$ is strictly included in
$D_0({\bf g}_1)$ and there cannot be system poles of $\mathbf{f}$ in
$D_0({\bf g}_1)\setminus \overline{D_0({\bf f})}$.

It is easy to see that
$$
\sum_{k=1}^{|{\bf m}|}
c_{j,k}^1\,\frac{P_{n,\mathbf{m},k}}{Q_{n,{\bf m}}}
$$
 is an $(n,|{\bf m}|,1)$
incomplete Pad\'e approximant of $g_{1,j}$. Using Theorem
\ref{inverso} with $m^*=1$, for each $j=1,\ldots,|{\bf m}|- n_1^*$,
either the disk $D_1( g_{1,j})$ contains exactly one pole of
$g_{1,j}$, and it is a zero of $Q_{|{\bf m}|}$, or $R_0(Q_{|{\bf
m}|}g_{1,j}) > R_1(g_{1,j})$. In particular, $D_0({\bf
g}_1)\not=\mathbb{C}$ and all the singularities of ${\bf g}_1$ on the
boundary of
$D_0({\bf g}_1)$ are poles which are zeros of $Q_{|{\bf m}|}$ counting
their order.
They constitute the next layer of candidate system poles of
$\mathbf{f}$ (now, it is possible that some candidates are not poles
of ${\bf f}$ since the functions $f_k$ intervene in the linear
combination as we saw in example \eqref{sistema1}). All the
system poles of $\mathbf{f}$ on the boundary of $D_0({\bf g}_1)$ must
necessarily be poles of $\mathbf{g}_1$ for the same reason as in the
preceding case.

 Let us denote these new candidates by $a_{n_1+1},\ldots,a_{n_1+n_2}$. Of
 course $n_1 +
n_2 \leq |{\bf m}|$. Should $n_1 + n_2 = |{\bf m}|$, we are done.
Otherwise, $n_2 < |{\bf m}| - n_1\le |{\bf m}| - n_1^*$ and we can
repeat the process. In order to eliminate the $n_2$ poles we have
$|{\bf m}| - n_1^*$ functions which are analytic on $D_0({\bf g}_1)$
and meromorphic on a neighborhood of $\overline{D_0({\bf g}_1)}$.
The corresponding homogeneous linear system of equations, similar to
\eqref{ecuacion},
has at least $|{\bf m}| - n_1^*-n_2$ linearly independent solutions
${\bf c}^2_j, j=1,\ldots,|{\bf m}| - n_1^* -n_2^*,\, n_2^*\le n_2.$
Set
\[ {\bf c}^2_j = (c_{j,1}^2,\ldots,c_{j,|{\bf m}|-n_1^*}^2),
\quad j=1,\dots,|{\bf m}| - n_1^*-n_2^*.
\]
Construct the $(|{\bf m}|-n_1^*-n_2^*) \times |{\bf m}|-n_1^*$
dimensional matrix
\[ C^2 = \left(
\begin{array}{c}
{\bf c}^2_1 \\
\vdots \\
{\bf c}^2_{|{\bf m}|-n_1^*-n_2^*}
\end{array}
\right).
\]
Define the system $\mathbf{g}_2$ of $|{\bf m}| -n_1^* -n_2^*$
functions by means of
\[ {\bf g}_2^t=
C^2 {\bf g}_1^t = C^2C^1{\bf f}^t = (g_{2,1},\ldots,g_{2,|{\bf
m}|-n_1^*-n_2^*})^t.
\]
The rows of $C^2C^1$ are of the form ${\bf c}_j^2 C^1, j=1,\ldots,
|{\bf m}| -n_1^*-n_2^*$, where $C^1$ has rank $|{\bf m}| -n_1^*$ and
the vectors ${\bf c}_k^2$ are linearly independent. Therefore, the
rows of $C^2C^1$ are linearly independent; in particular, they are
non-null. Consequently, the components of ${\bf g}_2$ are not
polynomials because of the algebraic independence of $\mathbf{f}$
with respect to $\mathbf{m}=(1,1,\dots,1)$. Thus, we can apply again
Theorem \ref{inverso}. The proof is completed using finite
induction.

Notice that the numbers $n_1, n_2, \ldots$ which so arise are
greater than or equal to $1$ and on each iteration their sum is less
than or equal to $|{\bf m}|$. Therefore, in a finite number of steps
their sum must equal $|{\bf m}|$. Consequently, the number of
candidate system poles of $\mathbf{f}$ in some disk, counting
their multiplicities, is exactly equal to $|{\bf m}|$ and they are precisely
the zeros of $Q_{|{\bf m}|}$ as we wanted to prove.

Now, suppose that there exists a candidate  system pole of
$\mathbf{f}$ that is not such or being a system pole has order smaller than
the multiplicity of the corresponding zero of $Q_{|\bf m|}$. Then, for some
$\alpha\in\mathbb{Z}_+$, we have
$$
\mathbf{g}_\alpha=(g_{\alpha,1},\dots,g_{\alpha,\nu}),\quad
\nu=|m|-n_0^*-n_1^*-\dots-n_\alpha^*,
$$
with $0\le n^*_j\le n_j,\,j=0,1,\dots,\alpha$, $n_0=0$, and
$\mu=|\mathbf{m}|-n_0-n_1-\dots-n_\alpha>0$ such that there exists a
point $a$ on the boundary of the region
$$
D_0(\mathbf{g}_\alpha)=\bigcap_{j=1}^\nu D_0(g_{\alpha,j})
$$
that is a pole of order $\tau$ of $g_{\alpha,j_0}$ for some
$j_0\in\{1,\dots,\nu\}$  but is not a system pole of $\mathbf{f}$ or
is one of order less than $\tau$. Let $a_{n_\alpha+1},
\dots,a_{n_{\alpha+1}},\,
n_{\alpha+1}\ge\tau,$ be the singularities of the functions
$g_{\alpha,j}$ on the boundary of $D_0(\mathbf{g}_\alpha)$ counting
multiplicities. We distinguish two cases. First, suppose that
$n_1+\dots+n_{\alpha+1}=|\mathbf{m}|$; then
$n_{\alpha+1}=\mu\le\nu$. All the functions $g_{\alpha,j}$ admit
meromorphic extension to a neighborhood of
$\overline{D_0(\mathbf{g}_\alpha)}$. We pose the problem of finding
coefficients $c_1,\dots,c_\nu$ such that
$$
\sum_{j=1}^\nu c_j g_{\alpha,j}
$$
is analytic on a neighborhood of
$\overline{D_0(\mathbf{g}_\alpha)}$. The problem consists in solving
a linear homogeneous system with $\mu$ equations and $\nu$ unknowns
similar to \eqref{ecuacion} but, due to the fact that the point $a$
is not a system pole of $\mathbf{f}$ or it is one of order less than
$\tau$, one of the equations may be written as a linear combination
of the others and we have at most $\mu-1$ equations, with
$\mu-1<\nu$. So, a non-trivial solution necessarily exists which
defines a function $g$ analytic on a neighborhood of
$\overline{D_0(\mathbf{g}_\alpha)}$ by means of
$$
g=\sum_{j=1}^\nu c_j g_{\alpha,j}=\sum_{k=1}^{|\mathbf{m}|} d_k
f_k,\quad d_k\in\mathbb{C},\quad k=1,\dots,|\mathbf{m}|.
$$
Following the same argument used in the process carried out to find
the candidate system poles of $\mathbf{f}$, we deduce that $g$
is not a polynomial. Now,
$$
\sum_{k=1}^{|{\bf m}|} d_{k}\,\frac{P_{n,\mathbf{m},k}}{Q_{n,{\bf
m}}}
$$
 is an $(n,|{\bf m}|,1)$ incomplete
Pad\'e approximant of $g$. Using Theorem \ref{inverso} with $m^*=1$,
either the disk $D_1( g)$ contains exactly one pole of $g$, and it
is a zero of $Q_{|{\bf m}|}$, or $R_0(Q_{|{\bf m}|}g) > R_1(g)$. But
both alternatives are impossible since all the zeros of
$Q_{|\mathbf{m}|}$ belong to $\overline{D_0(\mathbf{g}_\alpha)}$.
So, we have reached a contradiction.

In case that $n_1+\dots+n_{\alpha+1}<|\mathbf{m}|$ we are in the
middle of the process described above and now, when solving the
corresponding system of equations to eliminate the $n_{\alpha+1}$
poles, we obtain $n^*_{\alpha+1}<n_{\alpha+1}$ since, again, one of
the equations is redundant. This implies that, in the last step, say
$\beta$, when $n_1+\dots+n_{\beta}=|\mathbf{m}|$ we have
$|\mathbf{m}|-n_0-n_1-\dots-n_\beta=\mu<\nu=
|\mathbf{m}|-n^*_0-n^*_1-\dots-n^*_\beta$ reaching the same
contradiction as before. We have proved a posteriori that
$n_j^*=n_j,\, j=1,2,\dots$

Thus, the proof of the inverse-type result is complete. Also, we
have that $Q_{|\mathbf{m}|}\equiv
\mathcal{Q}_{|\mathbf{m}|}(\mathbf{f},\mathbf{m})$.

Let us prove now that a) implies b). Except for some details related
to the numbers $R_{\xi}(\mathbf{f},\mathbf{m})$, where $\xi$ is a
system pole of $\mathbf{f}$, the arguments are similar to those
employed in \cite{GS}. In spite of this, for completeness, we give
the entire proof.

For each $n\ge |\mathbf{m}|$, let $q_{n,\mathbf{m}}$ be the
polynomial $Q_{n,\mathbf{m}}$ normalized so that
\begin{equation}\label{norma}
\sum_{k=1}^{|\mathbf{m}|} |\lambda_{n,k}|=1,\qquad
q_{n,\mathbf{m}}(z) = \sum_{k=1}^{|\mathbf{m}|} \lambda_{n,k}z^k.
\end{equation}
Due to this normalization, the polynomials $q_{n,\mathbf{m}}$ are
uniformly bounded on each compact subset of $\mathbb{C}$.

Let $\xi$ be a system pole of order $\tau$ of $\mathbf{f}$  with
respect to ${\bf m}$. Consider a polynomial combination $g_{1}$ of
type \eqref{combination} that is analytic on a neighborhood of
$\overline{D}_{|\xi|}$ except for a simple pole at $z=\xi$ and
verifies that $R_1(g_1)=R_{\xi,1}(\mathbf{f},\mathbf{m})\,
(=r_{\xi,1}(\mathbf{f},\mathbf{m})) $. Then, we have
$$
g_{1}=\sum_{k=1}^{|\mathbf{m}|}p_{k,1} f_k,\quad \deg
p_{k,1}<m_k,\quad k=1,\dots,|\mathbf{m}|,
$$
and
$$
q_{n,\mathbf{m}}(z)\,h_1(z)-(z-\xi)\, \sum_{k=1}^{|\mathbf{m}|}
p_{k,1}(z)\, P_{n,\mathbf{m},k}(z)= A z^{n+1}+\dots,
$$
where $h_1(z)=(z-\xi)\,g_1(z)$. Hence, the function
$$
\frac{q_{n,\mathbf{m}}(z)\dst\,h_1(z)}{z^{n+1}}-\frac{z-\xi}{z^{n+1}}
\sum_{k=1}^{|\mathbf{m}|} p_{k,1}(z)\, P_{n,\mathbf{m},k}(z)
$$
is analytic on $D_1(g_1)$. Take $0 < r < R_{1}(g_1)$ and set
$\Gamma_r=\{z\in\mathbb{C}\,:\,|z|=r\}$. Using Cauchy's formula, we
obtain
$$
q_{n,\mathbf{m}}(z)h_1(z)-(z-\xi)\, \sum_{k=1}^{|\mathbf{m}|}
 p_{k,1}(z) P_{n,\mathbf{m},k}(z)=\frac{1}{2\pi
i}\int_{\Gamma_r}\frac{z^{n+1}}{\omega^{n+1}}
\frac{q_{n,\mathbf{m}}(\omega)\,h_1(\omega)}{\omega -z}\,d\omega,
$$
for all  $z$ with $|z|<r$, since $\deg \sum_{k=1}^{|\mathbf{m}|}
p_{k,1} P_{n,\mathbf{m},k}<n$. In particular, taking $z=\xi$ in the
above formula, we arrive at
\begin{equation}\label{cauchy}
q_{n,\mathbf{m}}(\xi)\,h_1(\xi)=\frac{1}{2\pi
i}\int_{\Gamma_r}\frac{\xi^{n+1}}{\omega^{n+1}}
\frac{q_{n,\mathbf{m}}(\omega)\,h_1(\omega)}{\omega -\xi}\,d\omega.
\end{equation}
Straightforward calculations lead to
$$
\limsup_{n\to\infty}\left|h_1(\xi)q_{n,\mathbf{m}}(\xi)\right|^{1/n}\le
\frac{|\xi|}{r}.
$$
Using that $h_1(\xi)\not=0$ and making $r$ tend to $R_{1}(g_1)$ we obtain
$$
\limsup_{n\to\infty}\left|q_{n,\mathbf{m}}(\xi)\right|^{1/n}\le
\frac{|\xi|}{R_{\xi,1}(\mathbf{f},\mathbf{m})}
 <1.
$$

Now, we employ induction. Suppose that
\begin{equation}\label{derivada}
\limsup_{n\to\infty}\left|q_{n,\mathbf{m}}^{(j)}(\xi)\right|^{1/n}\le
\frac{|\xi|}{R_{\xi,j+1}(\mathbf{f},\mathbf{m})},\quad j=0,1,\dots,s-2
\end{equation}
(recall that $R_{\xi,j+1}(\mathbf{f},\mathbf{m}) =
\min_{k=1,\ldots,j+1}r_{\xi,k}(\mathbf{f},\mathbf{m})$), with
$s\le\tau$, and let us prove that formula \eqref{derivada} holds for
$j=s-1$.

Consider a polynomial combination $g_{s}$ of
the type \eqref{combination} that is analytic on a neighborhood of
$\overline{D}_{|\xi|}$ except for a pole of order $s$ at $z=\xi$ and
verifies that $R_s(g_s)=r_{\xi,s}(\mathbf{f},\mathbf{m})$. Then, we
have
$$
g_{s}=\sum_{k=1}^{|\mathbf{m}|}p_{k,s} f_k,\quad \deg
p_{k,s}<m_k,\quad k=1,\dots,|\mathbf{m}|.
$$
Set $h_s(z)= (z-\xi)^s g_s(z)$. Reasoning as in the previous case,
the function
$$
\frac{q_{n,\mathbf{m}}(z)\dst\,h_s(z)}{z^{n+1}(z-\xi)^{s-1}}-
\frac{z-\xi}{z^{n+1}}
\sum_{k=1}^{|\mathbf{m}|} p_{k,s}(z)\, P_{n,\mathbf{m},k}(z)
$$
is analytic on $D_s(g_s)\setminus \{\xi\}$. Put
$P_s=\sum_{k=1}^{|\mathbf{m}|} p_{k,s} P_{n,\mathbf{m},k}$. Fix an
arbitrary compact set $K\subset (D_s(g_s)\setminus \{\xi\})$.  Take
$\delta>0$ sufficiently small and $0 < r < R_{s}(g_s)$ with
$K\subset D_r$. Using Cauchy's integral formula and the residue
theorem, for all $z\in K$, we have
\begin{equation}\label{hermite}
\frac{q_{n,\mathbf{m}}(z)h_s(z)}{(z-\xi)^{s-1}}-(z-\xi)\,P_s(z)
=I_n(z)-J_n(z),
\end{equation}
where
$$
I_n(z)=\frac{1}{2\pi i}\int_{\Gamma_r}\frac{z^{n+1}}{\omega^{n+1}}
\frac{q_{n,\mathbf{m}}(\omega)\,h_s(\omega)}{(\omega-\xi)^{s-1}(\omega
-z)}\,d\omega
$$
and
$$
J_n(z)=\frac{1}{2\pi
i}\int_{|\omega-\xi|=\delta}\frac{z^{n+1}}{\omega^{n+1}}
\frac{q_{n,\mathbf{m}}(\omega)\,h_s(\omega)}{(\omega-\xi)^{s-1}(\omega
-z)}\,d\omega.
$$
We have used in \eqref{hermite} that $\deg P_s<n$. The first
integral $I_n$ is estimated as in \eqref{cauchy}  to obtain
\begin{equation}\label{estimate1}
\limsup_{n\to\infty}
\left\|I_n(z)\right\|_K^{1/n}\le\frac{\|z\|_K}{R_s(g_s)}=
\frac{\|z\|_K}{r_{\xi,s}(\mathbf{f},\mathbf{m})}.
\end{equation}

As for $J_n$, write
$$
q_{n,\mathbf{m}}(\omega)=\sum_{j=0}^{|\mathbf{m}|}
\frac{q_{n,\mathbf{m}}^{(j)}(\xi)}{j!}(\omega-\xi)^{j}.
$$
Then
\begin{equation}\label{estimate2}
J_n(z)=\sum_{j=0}^{s-2}\frac{1}{2\pi
i}\int_{|\omega-\xi|=\delta}\frac{z^{n+1}}{\omega^{n+1}}
\frac{q_{n,\mathbf{m}}^{(j)}(\xi)}{j!(\omega -z)}
\frac{h_s(\omega)}{(\omega-\xi)^{s-1-j}}\,d\omega.
\end{equation}
Using the inductive hypothesis \eqref{derivada}, estimating the
integral in \eqref{estimate2}, and making $\varepsilon$ tend to
zero, we obtain
$$
\limsup_{n\to\infty}
\left\|J_n(z)\right\|_K^{1/n}\le\frac{\|z\|_K}{|\xi|}\frac{|\xi|}
{R_{\xi,s-1}(\mathbf{f},\mathbf{m})}=
\frac{\|z\|_K}{R_{\xi,s-1}(\mathbf{f},\mathbf{m})},
$$
which, together with \eqref{estimate1} and \eqref{hermite}, gives
\begin{equation}\label{estimate3}
\limsup_{n\to\infty} \left\| q_{n,\mathbf{m}}(z)h_s(z)-(z-\xi)^s\,
P_s(z)\right\|_K^{1/n}\le
\frac{\|z\|_K}{R_{\xi,s}(\mathbf{f},\mathbf{m})}.
\end{equation}

As the function inside the norm in \eqref{estimate3} is analytic
in
$D_s(g_s)$, inequality \eqref{estimate3} also holds for any compact
set $K\subset D_s(g_s)$. Besides, we can differentiate $s-1$ times
that function and the inequality still holds true by virtue of
Cauchy's integral formula. So, taking $z=\xi$ in \eqref{estimate3}
for the differentiated version, we obtain
$$
\limsup_{n\to\infty}\left|\left(q_{n,\mathbf{m}}h_s\right)^{(s-1)}(\xi)
\right|^{1/n}\le \frac{|\xi|}{R_{\xi,s}(\mathbf{f},\mathbf{m})}.
$$
Using  the Leibnitz formula for higher derivatives of a product of
two functions and the
induction  hypothesis \eqref{derivada}, we arrive at
$$
\limsup_{n\to\infty}\left|q_{n,\mathbf{m}}^{(s-1)}(\xi)
\right|^{1/n}\le \frac{|\xi|}{R_{\xi,s}(\mathbf{f},\mathbf{m})},
$$
since $h_s(\xi)\not=0$. This completes the induction.

Let $\xi_1,\dots,\xi_p$ be the distinct system poles of ${\bf f}$ and
let $\tau_i$ be the order of $\xi_i$ as a system pole, $i=1,\dots,p.$
By assumption, $\tau_1+\dots+\tau_p=|\mathbf{m}|$. We have proved
that, for $i=1,\dots,p$ and $j=0,1,\dots,\tau_i-1,$
\begin{equation}\label{derivada2}
\limsup_{n\to\infty}\left
|q_{n,\mathbf{m}}^{(j)}(\xi_i)\right|^{1/n}\le
\frac{|\xi_i|}{R_{\xi_i,j+1}(\mathbf{f},\mathbf{m})}\le
\frac{|\xi_i|}{R_{\xi_i}(\mathbf{f},\mathbf{m})}.
\end{equation}

Recall that $\mathcal{Q}_{|\mathbf{m}|}(\mathbf{f},\mathbf{m})$ is
the monic polynomial whose zeros are the system poles of
$\mathbf{f}$ with respect to $\mathbf{m}$. Denote by $L_{i,j},\,
i=1,\dots,p ;\, j=0,1,\dots \tau_i-1,$ the fundamental interpolating
polynomials at the zeros of
$\mathcal{Q}_{|\mathbf{m}|}(\mathbf{f},\mathbf{m})$; that is, for
each $i=1,\dots,p$ and $j=0,1,\dots \tau_i-1,$  $\deg L_{i,j}\le
|\mathbf{m}|-1$ and
$$
L_{i,j}^{(\nu)}(b_\kappa)=\delta_{i\kappa}\delta_{j\nu},\quad
\kappa=1,\dots,p,\quad \nu=0,1,\dots,\tau_i-1.
$$
Then
\begin{equation}\label{derivada3}
q_{n,\mathbf{m}}(z)=\lambda_{n,|\mathbf{m}|}
\mathcal{Q}_{|\mathbf{m}|}(\mathbf{f},\mathbf{m})+
\sum_{i=1}^p\sum_{j=0}^{\tau_i-1} q_{n,\mathbf{m}}^{(j)}(\xi_i)\,
L_{i,j}(z).
\end{equation}
From \eqref{derivada2} and  \eqref{derivada3} it follows that
$$
\limsup_{n\to\infty}\|q_{n,\mathbf{m}}-\lambda_{n,|\mathbf{m}|}
\mathcal{Q}_{|\mathbf{m}|}(\mathbf{f},\mathbf{m})\|_K^{1/n} \le
\theta<1,
$$
for any compact $K\subset\mathbb{C}$, where
\begin{equation}\label{speed}
\theta=\max\left\{\frac {|\xi|}
{R_{\xi}(\mathbf{f},\mathbf{m})}\,:\, \xi \in
{\mathcal{P}}_{|\mathbf{m}|}(\mathbf{f},\mathbf{m}) \right\}.
\end{equation}
As all norms in finite dimensional spaces are equivalent, we obtain
\begin{equation}\label{speed2}
\limsup_{n\to\infty}\|q_{n,\mathbf{m}}-\lambda_{n,|\mathbf{m}|}
\mathcal{Q}_{|\mathbf{m}|}(\mathbf{f},\mathbf{m})\|^{1/n} \le
\theta<1.
\end{equation}
Now, necessarily we have
\begin{equation} \label{des}
\liminf_{n\to\infty} |\lambda_{n,|\mathbf{m}|}|>0,
\end{equation}
since if there exists a subsequence $\Lambda\subset\mathbb{N}$ such
that $\lim_{n\in\Lambda} \lambda_{n,|\mathbf{m}|}=0$, then from
\eqref{speed2} we have $\lim_{n\in\Lambda} \|q_{n,\mathbf{m}}\|=0$,
contradicting \eqref{norma}.

As $q_{n,\mathbf{m}}=\lambda_{n,|\mathbf{m}|}Q_{n,\mathbf{m}}$, we
have proved
\begin{equation}\label{speed3}
\limsup_{n\to\infty}\|Q_{n,\mathbf{m}}-
\mathcal{Q}_{|\mathbf{m}|}(\mathbf{f},\mathbf{m})\|^{1/n} \le
\theta<1,
\end{equation}
 where $\theta$ is given by \eqref{speed}. In particular, for $n\ge
 n_0$, $\deg Q_{n,\mathbf{m}}=|\mathbf{m}|$. The difference of any
 two non-collinear solutions $Q_1$ and $Q_2$ of  Definition
\ref{defsimultaneos} with the same degree and equal leading
coefficient produces a new solution of smaller degree but we have
proved that any solution must have degree $|\mathbf{m}|$ for all sufficiently large $n$. Hence, the
polynomial
$Q_{n,\mathbf{m}}$ is uniquely determined for all sufficiently large
$n$. With this we have concluded the proof of the direct result.

Let us prove that the upper bound  in \eqref{speed3} actually
gives the exact rate of convergence to obtain \eqref{eq:6b}. To the
contrary, suppose that
\begin{equation}\label{speed4}
\limsup_{n\to\infty}\|Q_{n,\mathbf{m}}-
\mathcal{Q}_{|\mathbf{m}|}(\mathbf{f},\mathbf{m})\|^{1/n}=\theta^\prime<
 \theta.
\end{equation}

Let $\zeta$ be a system pole of $\mathbf{f}$ such that
$$
\frac {|\zeta|}
{R_{\zeta}(\mathbf{f},\mathbf{m})}=\theta=\max\left\{\frac {|\xi|}
{R_{\xi}(\mathbf{f},\mathbf{m})}\,:\, \xi \in
{\mathcal{P}}_{|\mathbf{m}|}(\mathbf{f},\mathbf{m}) \right\}.
$$
Naturally, if there is inequality in \eqref{speed4} then
$R_{\zeta}(\mathbf{f},\mathbf{m})<\infty$.

Choose a polynomial combination
\begin{equation}\label{combination2}
g=\sum_{k=1}^d p_k f_k,\quad \deg p_k<m_k,\quad k=1,\dots,d,
\end{equation}
that is analytic on a neighborhood of $\overline{D}_{|\zeta|}$
except for a pole of order $s$ at $z=\zeta$ with $R_s(g)=
R_{\zeta}(\mathbf{f},\mathbf{m})$. On the boundary of $D_s(g)$ the
function $g$ must have a singularity which is not a system pole. In
fact, if all the singularities were of this type we could find a
different polynomial combination $g_1$ of type  \eqref{combination2}
for which $R_s(g_1) > R_s(g)= R_{\zeta}(\mathbf{f},\mathbf{m})$
against our definition of $R_{\zeta}(\mathbf{f},\mathbf{m})$. For
short, put
$\mathcal{Q}_{|\mathbf{m}|}(\mathbf{f},\mathbf{m})=Q_{|\mathbf{m}|}$.
Consequently, the function $Q_{|\mathbf{m}|}\, g$ can be represented
as a power series $\sum_{j=0}^\infty c_jz^j$ with radius of
convergence $ R_{\zeta}(\mathbf{f},\mathbf{m})$. So
\begin{equation}\label{combination3} \limsup_{n\to\infty}
\sqrt[n]{|c_{n}|}=1/R_{\zeta}(\mathbf{f},\mathbf{m}).
\end{equation}

On the other hand, by virtue of \eqref{combination2}, we have
$$
H_n(z):=
 Q_{n,\mathbf{m}}(z)\,g(z)-\sum_{k=1}^d p_k(z)\,
P_{n,\mathbf{m},k}(z) = B_n z^{n+1}+\dots
$$
and this function is analytic at least in  $D_{|\zeta|}$ with a zero
of multiplicity at least $n+1$ at $z=0$.
Taking $r<|\zeta|$, we obtain
$$
\frac{1}{2\pi i}\int_{\Gamma_r}\frac{H_n(\omega)}{\omega^{n+1}}\,
d\omega =0.
$$
Set $P_n=\sum_{k=1}^d p_k P_{n,\mathbf{m},k}$. Clearly, $
Q_{|\mathbf{m}|}\, g \equiv {( Q_{|\mathbf{m}|}-Q_{n,\mathbf{m}})\,
g+P_n}
 +H_n
$
and, since $\deg P_n\le n-1$, we arrive at
$$
c_n=\frac{1}{2\pi i}\int_{\Gamma_r}\frac{Q_{|\mathbf{m}|}(\omega)\,
g(\omega)}{\omega^{n+1}}\, d\omega=\frac{1}{2\pi i}\int_{\Gamma_r}
\frac{\left[Q_{|\mathbf{m}|}(\omega)-Q_{n,\mathbf{m}}
(\omega)\right] g(\omega)} { \omega^{n+1}}  d\omega.
$$
Taking \eqref{combination3} and \eqref{speed4} into consideration,
estimating the integral, and letting $r$ tend to $|\zeta|$, it
follows that
$$
\frac{1}{R_{\zeta}(\mathbf{f},\mathbf{m})}=\limsup_{n\to\infty}
\sqrt[n]{|c_{n}|}\le\frac{\theta^\prime}{|\zeta|}<\frac{\theta}{|\zeta|}
=\frac{1}{R_{\zeta}(\mathbf{f},\mathbf{m})},
$$
which is absurd. We have completed the proof of Theorem \ref{reciprocal}. \hfill $\Box$

\subsection{Convergence of the Hermite-Pad\'e approximants}\label{convergence}

The following result is in some sense the analogue of the formula
displayed just after (58) in \cite{gon2} written in different terms.

\begin{thm} \label{cor:1} Assume that either ${\rm a})$ or ${\rm b})$ in
Theorem \ref{reciprocal} takes place. If $\xi$ is a system pole of
order $\tau$ of ${\bf f}$ with respect to $\bf m$, then
\begin{equation}\label{derivada2*}
\max_{j=0,\ldots,\overline{s}}\limsup_{n\to\infty}\left
|Q_{n,\mathbf{m}}^{(j)}(\xi)\right|^{1/n} =
\frac{|\xi|}{R_{\xi,\overline{s}+1}(\mathbf{f},\mathbf{m})},\quad
\overline{s}=0,1,\dots,\tau-1.
\end{equation}
\end{thm}

\begin{proof}
Let $\xi$ be as indicated.  From \eqref{derivada2} and \eqref{des}
we have
\[\max_{j=0,\ldots,\overline{s}}\limsup_{n\to\infty}\left
|Q_{n,\mathbf{m}}^{(j)}(\xi)\right|^{1/n} \leq
\frac{|\xi|}{R_{\xi,\overline{s}+1}(\mathbf{f},\mathbf{m})},\quad
\overline{s}=0,1,\dots,\tau-1.
\]
Assume that there is strict inequality for some $\overline{s}
\in \{0,\ldots,\tau-1\}$ and fix $\overline{s}$.

Choose a polynomial combination
$$
g=\sum_{k=1}^d p_k f_k,\quad \deg p_k<m_k,\quad k=1,\dots,d,
$$
that is analytic on a neighborhood of $\overline{D}_{|\xi|}$ except
for a pole of order $ {s}\,(\leq \overline{s}+1)$ at $z=\xi$ with
$R_s(g)= R_{\xi,\overline{s}+1}(\mathbf{f},\mathbf{m})$. As before,
on the boundary of $D_s(g)$ the function $g$ must have a singularity
which is not a system pole. Set
$\mathcal{Q}_{|\mathbf{m}|}(\mathbf{f},\mathbf{m})=Q_{|\mathbf{m}|}$.
Consequently, the function  $Q_{|\mathbf{m}|}\, g$ can be
represented as a power series $\sum_{j=0}^\infty c_jz^j$ with radius
of convergence $ R_{\xi,\overline{s}+1}(\mathbf{f},\mathbf{m})$. So
\begin{equation}\label{combination3*} \limsup_{n\to\infty}
\sqrt[n]{|c_{n}|}=1/R_{\xi,\overline{s}+1}(\mathbf{f},\mathbf{m}).
\end{equation}

On the other hand, by virtue of \eqref{combination2}, we have
$$
H_n(z):=
 Q_{n,\mathbf{m}}(z)\,g(z)-\sum_{k=1}^d p_k(z)\,
P_{n,\mathbf{m},k}(z) = B_n z^{n+1}+\dots
$$
and this function is analytic  in  $D_{s}(g) \setminus \{\xi\}$.
Take $r$ smaller than but sufficiently close to
$R_{\xi,\overline{s}+1}(\mathbf{f},\mathbf{m})$ and $\delta > 0$
sufficiently small. Let $\Gamma_{\delta,r}$ be the positively
oriented curve determined by $\gamma_{\delta}=\{\omega: |\omega -
\xi| = \delta\}$ and $\Gamma_r$. We have
$$
\frac{1}{2\pi
i}\int_{\Gamma_{\delta,r}}\frac{H_n(\omega)}{\omega^{n+1}}\, d\omega
=0.
$$
Set $P_n=\sum_{k=1}^d p_k P_{n,\mathbf{m},k}$ and
$h(\omega)=(\omega-\xi)^sg(\omega)$. Obviously, $$
Q_{|\mathbf{m}|}\, g\equiv {( Q_{|\mathbf{m}|}-Q_{n,\mathbf{m}})\,
g+P_n}
 +H_n
$$ and, since $\deg P_n\le n-1$, we obtain
\begin{multline*}\dst
c_n=\dst\frac{1}{2\pi
i}\int_{\Gamma_{\delta,r}}\frac{Q_{|\mathbf{m}|}(\omega)
g(\omega)}{\omega^{n+1}}\, d\omega =\frac{1}{2\pi
i}\int_{\Gamma_{\delta,r}} \frac{[Q_{|\mathbf{m}|}-Q_{n,\mathbf{m}}
](\omega) h(\omega)} {(\omega -\xi)^s \omega^{n+1}} d\omega
\\ \\ =\dst
\frac{1}{2\pi i}\int_{\Gamma_r }
\frac{[Q_{|\mathbf{m}|}-Q_{n,\mathbf{m}} ](\omega) h(\omega)}
{(\omega -\xi)^s \omega^{n+1}}  d\omega - \sum_{\nu = 0}^{|{\bf m}|}
\frac{1}{2\pi i}\int_{\gamma_\delta} \frac{
[Q_{|\mathbf{m}|}^{(\nu)} -Q_{n,\mathbf{m}}^{(\nu)}](\xi) h(\omega)}
{\nu!(\omega -\xi)^{s-\nu} \omega^{n+1}}  d\omega
\\ \\  =\dst
\frac{1}{2\pi i}\int_{\Gamma_r }
\frac{[Q_{|\mathbf{m}|}-Q_{n,\mathbf{m}} ](\omega) h(\omega)}
{(\omega -\xi)^s \omega^{n+1}}  d\omega - \sum_{\nu = 0}^{s-1}
\frac{1}{2\pi i}\int_{\gamma_\delta}
\frac{Q_{n,\mathbf{m}}^{(\nu)}(\xi) h(\omega)} {\nu!(\omega
-\xi)^{s-\nu} \omega^{n+1}}  d\omega.
\end{multline*}

Estimating these integrals, using \eqref{eq:6b} and the temporary
assumption that
\[ \max_{j=0,\ldots,\overline{s}}\limsup_{n\to\infty}\left
|Q_{n,\mathbf{m}}^{(j)}(\xi)\right|^{1/n} = \frac{|\xi|}{\kappa} <
\frac{|\xi|}{R_{\xi,\overline{s}+1}(\mathbf{f},\mathbf{m})},
\]
we obtain
\[ \limsup_{n\to \infty} |c_n|^{1/n} \leq \max \left\{\frac{1}{\kappa},
\frac{\theta}{R_{\xi,\overline{s}+1}(\mathbf{f},\mathbf{m})}\right\}
< \frac{ 1}{R_{\xi,\overline{s}+1}(\mathbf{f},\mathbf{m})},
\]
which contradicts \eqref{combination3*}. Hence,
\eqref{derivada2*} takes place.
\end{proof}

Now, we are ready to give the analogue of \eqref{eq:6} for
simultaneous approximation. We need to introduce some notation. Fix
$k\in\{1,\dots,d\}$. Let $D_{|{\bf m}|,k}({\bf f},{\bf m})$ be the
largest disk centered at $z=0$ in which all the poles of $f_k$ are
system poles of $\bf f$ with respect to $\bf m$, their order as
poles of $f_k$ does not exceed their order as system poles, and
$f_k$ has no other singularity. By $R_{|{\bf m}|,k}({\bf f},{\bf
m})$ we denote the radius of this disk. Let $\xi_1,\ldots,\xi_{N}$
be the poles of $f_k$ in $D_{|{\bf m}|,k}({\bf f},{\bf m})$. For
each $j=1,\dots,N$, let $\widetilde{\tau}_j$ be the order of $\xi_j$
as a pole of $f_k$ and $\tau_j$ its order as a system pole. By
assumption $\widetilde{\tau}_j\leq \tau_j$. Set
\[ R^*_{|{\bf m}|,k}({\bf f},{\bf m}) =
\min\left\{R_{|{\bf m}|,k}({\bf f},{\bf m}), \min_{j=1,\ldots,N}
R_{\xi_j,\widetilde{\tau}_j}(\bf f, m)\right\},
\]
and let $ D^*_{|{\bf m}|,k}({\bf f},{\bf m})$
be the disk centered at $z=0$ with this radius.

Recall that
$\sigma(B)$ stands for the $1$-dimensional Hausdorff content of the
set $B$. We say that a compact set $K\subset\mathbb{C}$ is
$\sigma$-regular if for each $z_0\in K$ and for each $\delta>0$, it
holds that $ \sigma\{z\in K\,:\, |z-z_0|<\delta \}>0. $

\begin{thm}\label{saff3}
Let ${\bf f}$ be a system of formal Taylor expansions as in
\eqref{sistema} and fix a multi-index $\mathbf{m}\in
{\mathbb{Z}}_+^d \setminus \{\bf 0\}$. Suppose that either ${\rm
a)}$ or ${\rm b)}$ in Theorem \ref{reciprocal} takes place. Then,
\begin{equation}\label{inequality3}
 \limsup_{n \to \infty}
\|f_k - R_{n,{\bf m},k}\|_{K}^{1/n} \leq \frac{\|z\|_
{K}}{R^*_{m_k}(f_k)}, \quad k=1,\ldots,d,
\end{equation}
where ${K}$ is any compact subset of $  D^*_{|{\bf m}|}(f_k)
\setminus {\mathcal{P}_{|\bf m|}(\bf f,\bf m)}$. If, additionally,
${K}$ is $\sigma$-regular, then we have equality in
\eqref{inequality3}. Moreover,
\[ R^*_{m_k}(f_k) = R^*_{|{\bf m}|,k}({\bf f},{\bf m}),
\qquad k=1,\ldots,d.
\]
\end{thm}

\begin{proof} Let us fix $k\in\{1,\dots,d\}$ and
maintain the notation introduced above. Let $K$ be a compact subset contained in
$D^*_{|{\bf m}|,k}({\bf f},{\bf m}) \setminus {\mathcal{P}_{|\bf
m|}(\bf f,\bf m)}$. Take $r$ smaller than but sufficiently close to
$R^*_{|{\bf m}|,k}({\bf f},{\bf m})$, and $\delta >0$ sufficiently
small so that $K$ is in the region bounded by $\Gamma_r$ and the
circles $\{z:|z - \xi_j| = \delta\}, j=1,\ldots,N$. Let
$\Gamma_{\delta,r}$ be the curve with positive orientation
determined by $\Gamma_r$ and those circles. On account of Definition
\ref{defsimultaneos}, using Cauchy's integral formula we have
\[ (Q_{n,\bf m}f_k - P_{n,{\bf m},k})(z) = \frac{1}{2\pi i}
\int_{\Gamma_{\delta,r}} \frac{z^{n+1}}{\omega^{n+1}}
\frac{(Q_{n,\bf m}f_k)(\omega)}{\omega - z} d\omega
\]
Since $\lim_n Q_{n,\bf m} = Q_{|{\bf m}|}$, using \eqref{derivada2*}
and standard arguments we obtain
\begin{equation} \label{final} \limsup_{n \to \infty}
\|f_k - R_{n,{\bf m},k}\|_{K}^{1/n} \leq
\frac{\|z\|_ {K}}{R^*_{|{\bf m}|,k}({\bf f},{\bf m})}.
\end{equation}

This last relation implies that $\sigma$-$\lim_{n\to \infty}
R_{n,{\bf m},k} = f_k$ inside $D^*_{|{\bf m}|,k}({\bf f},{\bf m})$.
Since $R_{m_k}^*(f_k)$ is the largest disk inside of which such
convergence takes place it readily follows that $R^*_{|{\bf
m}|,k}({\bf f},{\bf m}) \leq R_{m_k}^*(f_k)$. Should $D^*_{|{\bf
m}|,k}({\bf f},{\bf m})$ contain on its boundary some singularity
which is not a system pole then necessarily $R^*_{|{\bf m}|,k}({\bf
f},{\bf m}) = R_{m_k}^*(f_k)$ because $\sigma$-convergence implies
that all singularities inside must be zeros of $Q_{|{\bf m}|}$ but
the zeros of this polynomial are all system poles as we proved in
Theorem \ref{reciprocal}. Assume that $ R_{m_k}^*(f_k) > R^*_{|{\bf
m}|,k}({\bf f},{\bf m})$. Then, we have $R_{m_k}^*(f_k) >
\min_{j=1,\ldots,N} R_{\xi_j,\widetilde{\tau}_j}(\bf f, m)$. From
the proof of \cite[Theorem 3.6]{cacoq} we know that for each pole
$\xi$ of order $\widetilde{\tau}$ of $f_k$ inside $D_{m_k}^*(f_k)$
\[ \limsup_{n \to \infty} \left
|Q_{n,\mathbf{m}}^{(j)}(\xi)\right|^{1/n}\le
\frac{|\xi|}{R_{m_k}^*(f_k)},\quad
j=0,1,\dots,\widetilde{\tau}-1.
\]
This contradicts \eqref{derivada2*}. Consequently
$R_{m_k}^*(f_k) = R^*_{|{\bf m}|,k}({\bf f},{\bf m})$ as claimed.

Due to \eqref{final}, we have also proved \eqref{inequality3}. In
order to show that this formula is exact for $\sigma$-regular
compact subsets one must argue as in the corresponding part of the
proof of \cite[Theorem 4.4]{cacoq}.
\end{proof}

As compared with \cite[Theorem 4.4]{cacoq}, Theorem \ref{saff3}
offers weaker assumptions and a characterization of the values
$R_{m_k}^*(f_k)$ in terms of the analytic properties of the
functions in the system  instead of the coefficients of their Taylor
expansion. An open question is to obtain an analogous
characterization  when the assumptions of Theorem \ref{saff3} do not
take place.

It would be interesting to study inverse problems for row
sequences of Hermite-Pad\'e approximation when only the limit
behavior of some of the zeros of the polynomials $Q_{n,{\bf m}}$
is known, in the spirit of the conjectures proposed by A.A. Gonchar
in \cite{gon2}.


\end{document}